\documentclass[a4paper,11pt]{article}

\usepackage[english]{babel}
\usepackage[T1]{fontenc}
\usepackage[utf8]{inputenc}
\usepackage{enumerate}
\usepackage{amsmath}
\usepackage{amssymb}
\usepackage{amsthm}
\usepackage{mathtools}
\usepackage{color}
\usepackage{csquotes}
\usepackage{hyperref}
\usepackage[backend=biber,
            style=alphabetic,
            giveninits=true,
            isbn=false,
            doi=false,
            url=false,
            sorting=nyvt]{biblatex}
\renewbibmacro{in:}{}

\newtheorem{thm}{Theorem} 
\newtheorem{thmA}{Theorem} 
\newtheorem{lemma}{Lemma}

\newcommand{\dist}{\mathrm{dist}}
\newcommand{\supp}[1]{\mathrm{supp} \left({#1}\right)}
\newcommand{\naturals}{\mathbb{N}}
\newcommand{\integers}{\mathbb{Z}}
\newcommand{\reals}{\mathbb{R}}
\newcommand{\dyadic}{\mathcal{D}}
\newcommand{\halfspace}[1]{\reals^{#1 +1}_+}

\newcommand{\dilation}[2]{\mathrm{dil}_{\rho}({#1},{#2})} 
\newcommand{\hyperball}[2]{\textrm{B}_\rho({#1},{#2})}
\newcommand{\Lp}[1]{\mathrm{L}^{#1}}
\newcommand{\locLp}[1]{\mathrm{L}^{#1}_{loc}}
\newcommand{\BMO}{\mathrm{BMO}}
\newcommand{\bmo}{\mathbf{bmo}}
\newcommand{\continuous}{\mathcal{C}}
\newcommand{\schwartz}{\mathcal{S}}
\newcommand{\lip}[1]{\Lambda_{#1}}
\newcommand{\zyg}{\Lambda_{\ast}}
\newcommand{\besov}[3]{\mathrm{B}^{#1}_{{#2},{#3}}}
\newcommand{\triebel}[3]{\mathrm{F}^{#1}_{{#2},{#3}}}
\newcommand{\zygsemi}{\dot{\Lambda}_{\ast}}
\newcommand{\IBMO}[1]{\mathrm{I}_{#1}(\BMO)}
\newcommand{\Jbmo}[1]{\mathrm{J}_{#1}(\bmo)}

\newcommand{\norm}[2]{\left\lVert #1 \right\rVert_{#2}}
\newcommand{\poisson}[1]{\textrm{P}\left[{#1}\right]}
\newcommand{\fourier}[1]{\mathcal{F}\left[{#1}\right]}
\newcommand{\invfourier}[1]{\mathcal{F}^{-1}\left[{#1}\right]} 
\newcommand{\bessel}[1]{(1-\Delta)^{{#1}/2}}
\newcommand{\opbessel}[1]{\mathrm{J}^{#1}}

\title{Approximation in the Zygmund and Hölder classes on $\reals^n$}
\author{\thanks{First author is supported by The Finnish Academy grant 1309940.}
        Eero~Saksman and
        \thanks{Second author is supported by the Generalitat de Catalunya
                grant 2017 SGR 395, the Spanish Ministerio de Ciencia e Innovaci\'on
                projects MTM2014-51824-P and MTM2017-85666-P,
                and the European Research Council
                project CHRiSHarMa no. DLV-682402.}
        Odí~Soler~i~Gibert}
\date{}

\newcommand{\Addresses}{{
  \bigskip
  \footnotesize

  \noindent Eero~Saksman\\
  \textsc{University of Helsinki, Department of Mathematics and Statistics,\\
  P.O. Box 68 , FIN-00014 University of Helsinki, Finland,}\\
  \textit{E-mail address: } \texttt{Eero.Saksman@helsinki.fi}

  \medskip

  \noindent Odí~Soler~i~Gibert\\
  \textsc{Universität Würzburg, Institut für Mathematik\\
  97074 Würzburg, Germany}\\
  \textit{E-mail address: }\texttt{odi.solerigibert@mathematik.uni-wuerzburg.de}

}}

\begin{document}

  \maketitle
  
  \begin{abstract}
    We determine the distance (up to a multiplicative constant) in the Zygmund class $\zyg(\reals^n)$ to the subspace $\Jbmo{}(\reals^n).$
    The latter space is the image under the Bessel potential $J \coloneqq \bessel{-1}$ of the space $\bmo(\reals^n),$
    which is a non-homogeneous version of the classical $\BMO.$
    Locally, $\Jbmo{}(\reals^n)$ consists of functions that together with their first derivatives are in $\bmo(\reals^n).$
    More generally, we consider the same question when the Zygmund class is replaced
    by the Hölder space $\lip{s}(\reals^n),$ with $0 < s \leq 1$ and the corresponding subspace is $\Jbmo{s}(\reals^n),$
    the image under $\bessel{-s}$ of $\bmo(\reals^n).$
    One should note here that $\lip{1}(\reals^n) = \zyg(\reals^n).$
    Such results were known earlier only for $n = s = 1$ with a proof that does not extend to the general case.
    
    Our results are expressed in terms of second differences.
    As a byproduct of our wavelet based proof, we also obtain the distance from $f \in \lip{s}(\reals^n)$
    to $\Jbmo{s}(\reals^n)$ in terms of the wavelet coefficients of $f.$
    We additionally establish a third way to express this distance in terms of
    the size of the hyperbolic gradient of the harmonic extension of $f$ on the upper half-space $\halfspace{n}.$
    \medskip\\
    \noindent\textsc{MCS Class:} 26B35.
  \end{abstract}
  
  \section{Introduction}
  We say that a real valued function $f$ on $\reals^n$ is in the \emph{non-homogeneous
  Hölder class} of order $s,$ with $0 < s < 1,$ denoted $f \in \lip{s}(\reals^n)$
  or just $f \in \lip{s}$ when there is no ambiguity, if it is uniformly bounded and
  \begin{equation}
    \label{eq:HolderCondition}
    \sup_{x,\, 0 < |y| < 1} \frac{|f(x+y)-f(x)|}{|y|^s} + \norm{f}{\Lp{\infty}(\reals^n)} < \infty,
  \end{equation}
  where $|x|$ stands for the euclidean norm of $x \in \reals^n.$
  In order to define all smoothness classes $s > 0$ we first note that in the range $s \in (0,2)$
  the norm in $\lip{s}$ may be defined by using second differences:
  \begin{equation}
    \label{eq:HolderCondSecondDiff}
    \norm{f}{\lip{s}} \coloneqq \sup_{x,\, 0 < |y| < 1} \frac{|f(x+y) -2f(x) + f(x-y)|}{|y|^s}
                                + \norm{f}{\Lp{\infty}(\reals^n)},
  \end{equation}
  and the norms~\eqref{eq:HolderCondition} and~\eqref{eq:HolderCondSecondDiff} are equivalent for $s \in (0,1).$
  In general, if $s > 0$ and $s = m + t,$ with $m \in \naturals_0$ and $0 < t \leq 1,$ we define
  the Hölder class $\lip{s}$ as the space of bounded functions that are $m$ times continuously differentiable
  with all their derivatives of order $m$ belonging to $\lip{t}.$
  The norm of $f \in \lip{s}$ may be defined as
  $\norm{f}{\lip{s}} \coloneqq \sup_{|\alpha| = m} \norm{\partial^\alpha f}{\lip{t}} + \norm{f}{\Lp{\infty}}(\reals^n),$
  which is equivalent to~\eqref{eq:HolderCondSecondDiff} for $s \in (0,2).$
  
  In the case $s = 1$ we obtain the so called \emph{non-homogeneous Zygmund class}.
  This space is often denoted by $\zyg(\reals^n)$ or just $\zyg$
  -- thus as Banach spaces $\zyg(\reals^n) = \lip{1}(\reals^n)$ --
  and we denote the corresponding seminorm by
  \begin{equation}
    \label{eq:ZygmundSeminorm}
    \norm{f}{\zygsemi} \coloneqq \sup_{x,\, 0 < |y| < 1} \frac{|f(x+y) - 2f(x) + f(x-y)|}{|y|}.
  \end{equation}
  If $f$ is a polynomial of degree at most $1,$ then we have that $\norm{f}{\zygsemi} = 0.$
  Note that the norm in $\zyg$ is given by $\norm{f}{\zyg} = \norm{f}{\zygsemi} + \norm{f}{\Lp{\infty}}.$
  
  The Zygmund class is the natural definition of the H\"older classes $\lip{s}$ for $s = 1$
  as is suggested by~\eqref{eq:HolderCondSecondDiff},
  and this definition is more properly validated  by the coincidence $\lip{s} = \besov{s}{\infty}{\infty}$ for all $s > 0,$
  where $\besov{s}{p}{q}$ stands for the classical Besov space.
  The classes $\lip{s}$ have numerous applications, e.g. to PDE's and polynomial approximation or Calderón-Zygmund theory,
  and have been extensively studied (see for instance \cite{ref:Zygmund}, \cite[Chapter~V]{ref:Stein}, \cite{ref:TriebelFunctionSpaces-I},
  \cite{ref:MakarovSmoothMeasures}, \cite{ref:DonaireLlorenteNicolau}).
  
  From now on, we will restrict ourselves to the spaces $\lip{s}$ for $0 < s \leq 1,$
  although we could also formulate our results for all $s > 0,$
  but this would not bring anything essentially new apart from some additional non-interesting technicalities.
  
  A locally integrable function $f$ on $\reals^n$ is said to have \emph{bounded mean oscillation,}
  $f \in \BMO(\reals^n)$ or just $f \in \BMO,$ if
  \begin{equation*}
    \norm{f}{\BMO} \coloneqq \sup_Q \left(\frac{1}{|Q|} \int_Q |f(x)-f_Q|^2 \, dx\right)^{1/2} < \infty,
  \end{equation*}
  where $Q$ ranges over all finite cubes with sides parallel to the axes in $\reals^n$ and where
  \begin{equation*}
    f_Q \coloneqq \frac{1}{|Q|} \int_Q f(x)\, dx
  \end{equation*}
  is the average of $f$ on $Q.$
  We refer to \cite{ref:SteinHarmonicAnalysis} for basic facts on the space $\BMO.$
  
  Actually, as we will in general deal with non-homogeneous spaces,
  we are more interested in the \emph{non-homogeneous $\BMO$ space},
  denoted by $\bmo(\reals^n)$ following Triebel's convention.
  To be precise, consider the collection $\dyadic_0$ of dyadic cubes in $\reals^n$ of side-length 1,
  that is, the set of cubes $Q$ of the form
  \begin{equation*}
      Q = \{x \in \reals^n\colon x-k \in [0,1]^n\}
  \end{equation*}
  for some $k \in \integers^n.$
  The Banach space $\bmo(\reals^n)$ consists of all functions $f \in \BMO(\reals^n)$
  that are locally uniformly in $\Lp{2},$ i.e. those functions such that
  \begin{equation}
    \label{eq:bmoNorm}
    \norm{f}{\bmo} \coloneqq \norm{f}{\BMO}
                             + \sup_{Q \in \dyadic_0} \left(\int_Q |f(x)|^2\, dx\right)^{1/2}
                             < \infty.
  \end{equation}
  We also need to define the scale of non-homogeneous $\bmo$-Sobolev spaces.
  We first fix our convention for the Fourier transform.
  For a given function $f \in \Lp{1}(\reals^n),$ we define its Fourier transform $\fourier{f} = \hat{f}$ using the convention
  \begin{equation*}
      \fourier{f}(\xi) \coloneqq \int_{\reals^n} f(x) e^{-i x\cdot\xi}\, dx, \quad \xi \in \reals^n.
  \end{equation*}
  Given $r \in \reals,$ the standard lift operator $\opbessel{r}$ is defined via the Bessel potential:
  for a Schwartz distribution $f$ one sets
  \begin{equation*}
    \opbessel{r} f(x) \coloneqq \invfourier{(1+|\xi|^2)^{-r/2} \fourier{f}(\xi)}(x).
  \end{equation*}
  Due to the lift property of this operator one has e.g. for the Hölder spaces the relation
  $\opbessel{r} \lip{s}(\reals^n) = \lip{s+r}(\reals^n)$ for all $s,r > 0,$
  and this equality may be used to define the space $\lip{s}(\reals^n)$ also for $s \leq 0.$
  The Banach space $\Jbmo{s}(\reals^n)$ consists of all functions $f \in \locLp{2}(\reals^n)$
  such that $\opbessel{-s}f \in \bmo(\reals^n),$ and we set
  \begin{equation*}
    \norm{f}{\Jbmo{s}} \coloneqq \norm{\opbessel{-s}f}{\bmo}.
  \end{equation*}
  Here, if one replaces the condition $f \in \locLp{2}$ by $f \in \schwartz',$
  the above definition extends to all $s \in \reals.$
  As we will see later on, for each $s > 0$ we have the strict inclusion
  \begin{equation*}
    \Jbmo{s} \subsetneq \lip{s}.
  \end{equation*}
  
  In \cite{ref:NicolauSoler} the authors give an estimate in terms of second differences
  for the distance of a compactly supported function $f \in \zyg$
  to the subspace $\IBMO{1}(\reals)$
  consisting of continuous functions with distributional derivative in $\BMO$
  -- this space is the homogeneous counterpart of the space $\Jbmo{1}$ we defined previously.
  The result in \cite{ref:NicolauSoler} is analogous to that obtained by P.~Ghatage and D.~Zheng in \cite{ref:GhatageZheng},
  where they estimate the distance of a function in the Bloch space to the subspace of
  analytic $\BMO$ functions in terms of its hyperbolic derivative.
  In this connection we mention also the paper of J.~Garnett and P.~Jones \cite{ref:GarnettJonesDistanceInBMO},
  where a formula for the distance from $\BMO$ to bounded functions is established.
  
  In order to state the result in \cite{ref:NicolauSoler} and our generalisation,
  it is useful to define for a given continuous function $f$ on $\reals^n$
  its (maximal) \emph{second difference} at point $x \in \reals^n$ and scale $y > 0$ by the quantity
  \begin{equation*}
    \Delta_2f(x,y) \coloneqq \sup_{|h| = y} |f(x+h) - 2f(x) + f(x-h)|.
  \end{equation*}
  We will always consider these second differences to be defined on the upper half-space,
  that is $(x,y) \in \halfspace{n}.$
  Moreover, if $X \subset \lip{s}$ is a subspace of the Hölder class of order $s,$
  we define the distance of $f \in \lip{s}$ to $X$ as
  \begin{equation*}
    \dist_s(f,X) \coloneqq \inf_{g \in X} \norm{f-g}{\lip{s}}.
  \end{equation*}
  The estimate in \cite{ref:NicolauSoler} is given in terms of a Carleson type condition
  on the set where the quantities $\Delta_2 f(x,y)/y$ are large, for $x \in \reals$ and $y > 0.$
  Namely, for $f \in \zyg(\reals)$ and $\varepsilon > 0$ consider the set
  \begin{equation*}
    S(f,\varepsilon) \coloneqq \{(x,y) \in \reals^2_+\colon \Delta_2f(x,y) > \varepsilon y\}
  \end{equation*}
  in the upper half-plane and the quantity
  \begin{equation*}
      M(f,\varepsilon) \coloneqq
      \sup_I \frac{1}{|I|} \int_I\int_0^{|I|} \chi_{S(f,\varepsilon)}(x,y)\, \frac{dy\, dx}{y},
  \end{equation*}
  where $I$ ranges over all intervals.
  Recall here that
  \begin{equation*}
    \IBMO{1}(\reals) = \{f\in\continuous(\reals)\colon f' \in \BMO\},
  \end{equation*}
  where the derivative is taken in the distributional sense.
  Then, the following holds
  \begin{thmA}[\cite{ref:NicolauSoler}]
    \label{thm:ZygmundIBMONicolauSoler}
    Let $f$ be a compactly supported function in $\zyg(\reals).$ 
    Then,
    \begin{equation*}
      \dist_1 (f,\IBMO{1}) \simeq \inf \{\varepsilon > 0\colon M(f,\varepsilon) < \infty\}.
    \end{equation*}
  \end{thmA}
  \noindent Some remarks are in order here.
  First of all, the precise theorem in \cite{ref:NicolauSoler} uses only the homogeneous seminorm $\norm{\cdot}{\zygsemi}$
  when defining the distance of a function to the subspace in question.
  However, this makes no real difference since $f$ is compactly supported.
  Secondly, Theorem~\ref{thm:ZygmundIBMONicolauSoler} involves the homogeneous $\BMO$-Sobolev space,
  but one may observe that the non-homogeneous norm is equivalent for functions supported on a fixed compact subset.
  
  The proof in \cite{ref:NicolauSoler} makes use of R.~Strichartz's \cite{ref:Strichartz} characterisation
  of homogeneous $\BMO$-Sobolev spaces in terms of second differences and exhibits a rather intrincate transfer
  from dyadic spaces to the non-dyadic situation.
  This method appears to be restricted to dimension $n = 1.$
  
  In this paper, our goal is to generalise Theorem~\ref{thm:ZygmundIBMONicolauSoler} in two directions:
  first of all, we will treat functions on $\reals^n$ for arbitrary $n \geq 1$ and, secondly,
  we consider not just the Zygmund class but functions in all classes $\lip{s}(\reals^n)$ for $0 < s \leq 1.$
  As a byproduct we provide a new approach to Theorem~\ref{thm:ZygmundIBMONicolauSoler}.
  
  We now turn to precise formulation of the results of the present paper.
  Some of them will be expressed in terms of a Carleson type measure for the size of subsets of the upper half-space,
  analogously to Theorem~\ref{thm:ZygmundIBMONicolauSoler},
  and it is useful to have a general definition for this purpose.
  Consider the set $\dyadic$ of dyadic cubes of side-length at most one in $\reals^n,$
  that is the set of cubes of the form
  \begin{equation*}
      Q = \{x \in \reals^n\colon 2^jx-k \in [0,1]^n\}, \quad j \in \naturals_0, k \in \integers^n.
  \end{equation*}
  For a given measurable subset $A \subseteq \halfspace{n},$ we define the quantity $M(A)$ by
  \begin{equation}
    \label{eq:CarlesonSetMeasure}
    M(A) \coloneqq \sup_{Q \in \dyadic} \frac{1}{|Q|} \int_Q\int_0^{l(Q)} \chi_A(x,y)\, \frac{dy\, dx}{y}.
  \end{equation}
  Note that the finiteness of $M(A)$ is equivalent to $y^{-1} \chi_A(x,y)\, dx\, dy$ being a Carleson measure
  in the upper half-space (see \cite[Section~II.2]{ref:SteinHarmonicAnalysis}).
  
  Towards our first characterisation,
  given a function $f \in \lip{s}(\reals^n),$ with $0 < s \leq 1,$ consider the set
  \begin{equation*}
    S(s,f,\varepsilon) \coloneqq \{(x,y) \in \halfspace{n}\colon \Delta_2f(x,y) > \varepsilon y^s\}.
  \end{equation*}
  This should be thought as the set of points $(x,y)$ in the upper half-space for which its associated second
  difference is large with respect to the corresponding scale.
  The following result generalises Theorem~\ref{thm:ZygmundIBMONicolauSoler} for arbitrary dimension $n \geq 1$
  and smoothness in the range $0 < s \leq 1.$
  \begin{thm}
    \label{thm:DistanceJbmoDifferences}
    Let $0 < s \leq 1,$ and consider a function $f \in \lip{s}(\reals^n).$
    Then,
    \begin{equation}
      \label{eq:DistanceJbmoDifferences}
      \dist_s (f,\Jbmo{s}) \simeq \inf \{\varepsilon > 0\colon M(S(s,f,\varepsilon)) < \infty\}.
    \end{equation}
  \end{thm}
  
  A main tool for us will be the wavelet characterisation of the function spaces involved.
  For that end we next recall the basic properties of wavelets.
  Consider the space $\Lp{2}(\reals^n)$ of square integrable functions on $\reals^n.$
  It is known that, for any $r \in \naturals_0,$ there exist compactly supported real-valued functions
  \begin{equation*}
      \varphi \in \continuous_0^r(\reals^n) \quad
      \text{and} \quad
      \psi_l \in \continuous_0^r(\reals^n), \quad
      \text{with } 1 \leq l \leq 2^n-1,
  \end{equation*}
  such that their dyadic translations and dilations form an orthonormal basis of $\Lp{2}(\reals^n).$
  To be more precise, the set
  \begin{equation*}
    \{\varphi(x-k)\colon k \in \integers^n\}
    \cup
    \{2^{jn/2} \psi_l(2^jx-k)\colon 1 \leq l \leq 2^n-1, j \in \naturals_0, k \in \integers^n\}
  \end{equation*}
  forms an orthonormal basis of $\Lp{2}(\reals^n).$
  Moreover, we additionally have that
  \begin{equation*}
    \int_{\reals^n} x^\alpha \psi_l(x)\, dx = 0
  \end{equation*}
  for any multi-index $\alpha \in \naturals_0^n$ such that $|\alpha| \leq r$ and for every $1 \leq l \leq 2^n-1.$
  Such a set of functions is called a wavelet basis of regularity $r.$
  For a detailed explanation on how to construct such bases see, for instance, \cite{ref:MeyerWaveletsOperators}.
  
  It will be useful to index the wavelets in terms of dyadic cubes.
  Thus, recall that we denote by $\dyadic$ the set of dyadic cubes on $\reals^n$ of side-length at most one and,
  for $j \in \naturals_0,$ let us denote the set of cubes $Q \in \dyadic$ of side length $l(Q) = 2^{-j}$ by $\dyadic_j.$
  Assuming that $Q \in \dyadic_j,$ we let $\tau(Q) \coloneqq j$ denote the dyadic level of $Q.$
  Given $k \in \integers^n,$ let $Q = \{x \in \reals^n\colon x-k \in [0,1]^n\} \in \dyadic_0$
  and denote
  \begin{equation*}
    \varphi_Q(x) = \varphi(x-k).
  \end{equation*}
  Analogously, given $j \in \naturals_0$ and $k \in \integers^n,$
  let $Q = \{x \in \reals^n\colon 2^jx-k \in [0,1]^n\} \in \dyadic_j$ and write
  \begin{equation*}
    \psi_{(l,Q)}(x) = 2^{jn/2} \psi_l(2^jx-k)
  \end{equation*}
  for $1 \leq l \leq 2^n-1.$
  For future convenience, we define $\mathcal{Q} = \{(l,Q)\colon 1 \leq l \leq 2^n-1, Q \in \dyadic\}$ and,
  for $\omega = (l,Q) \in \mathcal{Q},$ we denote $|\omega| \coloneqq \tau(Q).$
  In addition, we use the notation $\mathcal{Q}_j = \{\omega \in \mathcal{Q}\colon |\omega| = j\}$ for $j \in \naturals_0.$
  Finally, if $Q \in \dyadic,$ then we denote by $\mathcal{Q}(Q)$ the set of $(l,P) \in \mathcal{Q}$
  for which $P \subseteq Q.$
  
  Consider a wavelet basis $\{\varphi_Q\colon Q \in \dyadic_0\} \cup \{\psi_\omega\colon \omega \in \mathcal{Q}\}$
  of regularity $r.$
  Let $f$ be a function in $\lip{s}$ for some $0 < s \leq 1.$
  The wavelet coefficients of $f$ are
  \begin{equation*}
      d_Q(f) = \int_{\reals^n} f(x) \overline{\varphi_Q(x)}\, dx, \quad Q \in \dyadic_0,
  \end{equation*}
  and
  \begin{equation*}
    c_{(l,Q)}(f) = c_\omega(f) \coloneqq \int_{\reals^n} f(x)\overline{\psi_\omega(x)}\, dx,
    \quad (l,Q) = \omega \in \mathcal{Q}.
  \end{equation*}
  From now on, since we only consider real-valued wavelets,
  we might omit the complex conjugation in the definition of the wavelet coefficients.
  Also, these coefficients can actually be defined for a much wider class of distributions,
  but we will not require this level of generality.
  In \cite{ref:LemarieMeyer}, P.~Lemarié and Y.~Meyer characterise when a wavelet series $f$ is in the space $\lip{s},$
  for $s > 0,$ in terms of its wavelet coefficients $\{d_Q(f)\}$ and $\{c_\omega(f)\}.$
  See also \cite{ref:AimarBernardisWaveletMeanOscillation} and \cite[Section~6.4]{ref:MeyerWaveletsOperators}
  for a more detailed explanation.
  \begin{thmA}[P.~Lemarié, Y.~Meyer]
    \label{thm:LipWavelet}
    Let $s > 0,$
    and consider a wavelet basis $\{\varphi_Q\colon Q \in \dyadic_0\} \cup \{\psi_\omega\colon \omega\in\mathcal{Q}\}$
    of regularity $r > s.$
    The wavelet series
    \begin{equation*}
      f(x) = \sum_{Q \in \dyadic_0} d_Q(f) \varphi_Q(x)
      + \sum_{\omega\in\mathcal{Q}} c_\omega(f) \psi_\omega(x), \quad x\in\reals^n,
    \end{equation*}
    is in $\lip{s}(\reals^n)$ if and only if
    \begin{equation}
      \label{eq:ZygmundWaveletCondition}
      \sup_{Q \in \dyadic_0} |d_Q(f)|
      + \sup_{\omega \in \mathcal{Q}} 2^{|\omega|(n/2+s)}|c_\omega(f)| < \infty.
    \end{equation}
    Moreover, if $\norm{\{d(f),c(f)\}}{s}$ is the left-hand side in~\eqref{eq:ZygmundWaveletCondition},
    then $\norm{\{d(f),c(f)\}}{s} \simeq \norm{f}{\lip{s}}.$
  \end{thmA}
  
  In \cite{ref:LemarieMeyer}, the authors also give a wavelet characterisation
  for functions in the space $\BMO$ (see also \cite{ref:AimarBernardisWaveletMeanOscillation},
  \cite[Section~5.6]{ref:MeyerWaveletsOperators} and \cite[Section~IV.4.5]{ref:SteinHarmonicAnalysis}
  for detailed expositions on the topic).
  However, we need the following wavelet characterisation for the non-homogeneous $\bmo$-Sobolev spaces.
  To our knowledge, such a characterisation using smooth wavelets appears first as a particular case of
  a theorem of M.~Frazier and B.~Jawerth in \cite{ref:FrazierJawerthDiscreteTransform}.
  \begin{thm}
    \label{thm:JbmoWavelet}
    Let $s > 0,$
    and consider a wavelet basis $\{\varphi_Q\colon Q \in \dyadic_0\} \cup \{\psi_\omega\colon \omega\in\mathcal{Q}\}$
    of regularity $r > s.$
    The wavelet series
    \begin{equation*}
      f(x) = \sum_{Q \in \dyadic_0} d_Q(f) \varphi_Q(x)
      + \sum_{\omega \in \mathcal{Q}} c_\omega(f) \psi_\omega(x), \quad x\in\reals^n,
    \end{equation*}
    represents an element in $\Jbmo{s}(\reals^n)$ if and only if
    \begin{equation}
    \label{eq:JBMOrWavelet}
        \sup_{Q \in \dyadic_0} |d_Q(f)|
        + \sup_{Q \in \dyadic} 
        \left(\frac{1}{|Q|} \sum_{\omega \in \mathcal{Q}(Q)} 4^{|\omega|s}|c_\omega(f)|^2\right)^{1/2} < \infty.
    \end{equation}
    Moreover, the above quantity is comparable to $\norm{f}{\Jbmo{s}}.$
  \end{thm}
  
  It turns out that the characterisation of the distance to $\Jbmo{s}$
  can be done in a rather simple way with the above notation in terms of wavelet coefficients.
  In particular, we make use of the unconditional convergence of the wavelet series
  appearing in Theorems~\ref{thm:LipWavelet} and~\ref{thm:JbmoWavelet},
  both in the spaces $\lip{s}(\reals^n)$ and in the spaces $\Jbmo{s}(\reals^n)$
  (of course, this convergence must be understood not in norm, but in the sense of distributions).
  Let $0 < s \leq 1,$ $\varepsilon > 0$ and fix a wavelet basis
  $\{\varphi_Q\colon Q \in \dyadic_0\} \cup \{\psi_\omega\colon \omega\in\mathcal{Q}\}.$
  Given a function $f \in \lip{s}(\reals^n)$ with wavelet coefficients $\{c_\omega(f)\},$
  we consider the set
  \begin{equation*}
    W(s,f,\varepsilon) \coloneqq
    \bigcup_{j \in \naturals_0} \{Q \in \dyadic_j\colon \sup_l |c_{(l,Q)}(f)| > \varepsilon 2^{-j(n/2+s)}\}
  \end{equation*}
  and the associated set $T(s,f,\varepsilon) \subseteq \halfspace{n}$ defined by
  \begin{equation*}
    T(s,f,\varepsilon) \coloneqq \bigcup_{Q \in W(s,f,\varepsilon)} T(Q),
  \end{equation*}
  where for a given cube $Q \in \dyadic$ we denote
  $T(Q) = \{(x,y) \in \halfspace{n}\colon x \in Q, l(Q)/2 \leq y \leq l(Q)\}.$
  Observe that $T(s,f,\varepsilon)$ comprises those top half-cubes in $\halfspace{n}$
  corresponding to cubes in $\dyadic$ having at least one associated wavelet coefficient $c_\omega(f)$ large
  with respect to its scale.
  Moreover, we emphasise that the sets $W(s,f,\varepsilon)$ and $T(s,f,\varepsilon)$ do not take into account
  the wavelet coefficients $\{d_Q(f)\}$ corresponding to function $\varphi.$
  \begin{thm}
    \label{thm:DistanceJbmoWavelets}
    Let $0 < s \leq 1$ and $\varepsilon > 0.$
    Consider a wavelet basis $\{\varphi_Q\colon Q \in \dyadic_0\} \cup \{\psi_\omega\colon \omega\in\mathcal{Q}\}$
    of regularity $r > s,$
    a function $f \in \lip{s}(\reals^n)$ and the corresponding set $T(s,f,\varepsilon)$ defined
    in terms of this wavelet basis.
    Then, we have that
    \begin{equation}
      \label{eq:DistanceJbmoWavelets}
      \dist_s(f,\Jbmo{s}) \simeq \inf \{\varepsilon>0\colon M(T(s,f,\varepsilon)) < \infty\}.
    \end{equation}
  \end{thm}
  \noindent Note that the infimum in~\eqref{eq:DistanceJbmoWavelets} is taken over a non empty set,
  since for any function $f \in \lip{s}(\reals^n)$ with wavelet coefficients $\{d_Q(f)\}$ and $\{c_\omega(f)\}$ one has that
  $M(T(s,f,\norm{\{d(f),c(f)\}}{s})) = 0,$ where
  \begin{equation*}
    \norm{\{d(f),c(f)\}}{s} \coloneq \sup_{Q \in \dyadic_0} |d_Q(f)|
    + \sup_{\omega \in \mathcal{Q}} 2^{|\omega|(n/2+s)}|c_\omega(f)|.
  \end{equation*}
  Of course, here the values of the coefficients $\{d_Q(f)\}$ do not play any role as long as they are uniformly bounded.
  This is because in this situation the series $\sum_{Q\in\dyadic_0} d_Q(f) \varphi_Q \in \Jbmo{s}$ for $s < r.$
  Therefore, to estimate $\dist_s(f,\Jbmo{s})$ we could assume that $d_Q(f) = 0$ for all $Q \in \dyadic_0$
  without loss of generality.
  Observe as well that the set $T(s,f,\varepsilon)$ might depend on the wavelet basis we choose.
  Nonetheless, this will have no consequence for our results since we will focus on the comparability
  between the infima in equations~\eqref{eq:DistanceJbmoDifferences} and~\eqref{eq:DistanceJbmoWavelets},
  and to that end we will consider a fixed wavelet basis of enough regularity.
  
  Another way to characterise when a continuous bounded function $f$ belongs to $\lip{s}(\reals^n),$
  for $0 < s \leq 1,$ is by means of the hyperbolic derivatives of its Poisson extension to the
  upper half-space.
  Namely, let us denote by $P_y(x)$ the Poisson kernel on the upper half-space $\halfspace{n},$ and by $u$
  the harmonic extension of $f,$ that is $u(x,y) = \poisson{f}(x,y) = (P_y \ast f)(x).$
  Given $0 < s \leq 1,$ a continuous function $f$ is in $\lip{s}(\reals^n)$ if and only if
  \begin{equation}
    \label{eq:LipSpacesHyperbolicDerivative}
    \norm{f}{\Lp{\infty}}
    + \sup_{(x,y)\in\halfspace{n}} y^{2-s} \left|\frac{\partial^2 u}{\partial y^2}(x,y)\right| < \infty.
  \end{equation}
  Moreover, the above quantity is comparable to $\norm{f}{\lip{s}}.$
  For a detailed exposition on the topic, see \cite[pp.~141--149]{ref:Stein}.
  This motivates us to estimate the distance of a given function $f \in \lip{s}(\reals^n)$
  to the subspace $\Jbmo{s}(\reals^n)$ in terms of these hyperbolic derivatives.
  Consider the set
  \begin{equation*}
    D(s,f,\varepsilon) = \left\{(x,y) \in \halfspace{n}\colon
    y^2\left|\frac{\partial^2 \poisson{f}}{\partial y^2}(x,y)\right| > \varepsilon y^s\right\},
  \end{equation*}
  that is the set of points in the upper half-space for which the second hyperbolic
  derivative of $f$ is large with respect to the corresponding scale.
  We have the following result.
  \begin{thm}
    \label{thm:DistanceJbmoDerivatives}
    Let $0 < s \leq 1,$ and consider a function $f \in \lip{s}(\reals^n).$
    Then,
    \begin{equation}
      \label{eq:DistanceJbmoDerivatives}
      \dist_s (f,\Jbmo{s}) \simeq \inf \{\varepsilon > 0\colon M(D(s,f,\varepsilon)) < \infty\}.
    \end{equation}
  \end{thm}
  
  Our proof of  Theorems~\ref{thm:DistanceJbmoDifferences} and~\ref{thm:DistanceJbmoDerivatives}
  will be via Theorem~\ref{thm:DistanceJbmoWavelets}.
  However, the reduction is rather non-trivial and will be based on careful comparison of
  the sets $S(s,f,\varepsilon),$ $D(s,f,\varepsilon)$ and $T(s,f,\varepsilon).$
  Our aim is to show that there are inequalities of the type  $M(T(s,f,\varepsilon)) \lesssim M(S(s,f,c\varepsilon)),$
  and similar inequalities between the other pairs, for an absolute constant $c > 0.$
  The proofs of these inequalities employ suitable inclusions between hyperbolically dilated sets.
  The inequalities then easily yield Theorems~\ref{thm:DistanceJbmoDifferences} and~\ref{thm:DistanceJbmoDerivatives}.
  
  The rest of the paper is structured as follows.
  Section~\ref{sec:WaveletCharacterisations} reduces Theorem~\ref{thm:JbmoWavelet} to known results in the literature
  and gives the proof of Theorem~\ref{thm:DistanceJbmoWavelets}.
  In Section~\ref{sec:PropertiesOfTheSets}, we first study the variability of the second differences,
  the wavelet coefficients and the hyperbolic derivative with respect to the location in the upper half-space.
  This is measured in terms of the hyperbolic distance.
  Notions related to the hyperbolic metric will be recalled in the beginning of that section.
  Finally, in Section~\ref{sec:Equivalences}, we are able to make a rigorous comparison of the sets $S,$ $D$ and $T.$
  The remaining details of Theorems~\ref{thm:DistanceJbmoDifferences} and~\ref{thm:DistanceJbmoDerivatives}
  are then given at the end of Section~\ref{sec:Equivalences}.
  
  \bigskip
  
  \noindent \textbf{Acknowledgements. } We would like to thank Óscar Domínguez, Eugenio Hernández, Oleg Ivrii and Artur Nicolau
  for helpful comments and conversations on the topic.
  We are also indebted to the referees for their valuable comments which have substantially improved
  the presentation of the paper.
  
  \bigskip
  
  \noindent \textbf{Notation. } We denote $\naturals_0 = \naturals \cup \{0\}.$
  For a measurable set $A \subset \reals^n,$ we denote by $|A|$ its Lebesgue measure,
  and we denote by $\chi_A$ its indicator function.
  In the particular case that the set $A$ is a cube, we denote by $l(A)$ its side-length.
  We use the standard notation $a \lesssim b$ (respectively $a \gtrsim b$)
  if there exists an absolute constant $C > 0$ such that $a \leq Cb$ (resp. $a \geq Cb$).
  We also denote $a \simeq b$ if $a \lesssim b$ and $a \gtrsim b.$
  
  If $\alpha = (\alpha_1,\ldots,\alpha_n) \in \naturals_0^n,$
  we say that $\alpha$ is a multi-index of length $|\alpha| = \alpha_1 + \ldots + \alpha_n,$
  and we use the notation $\partial^\alpha = (\partial/\partial x_1)^{\alpha_1} \dots (\partial/\partial x_n)^{\alpha_n}.$
  For a multi-index $\alpha$ and $x = (x_1, \ldots, x_n) \in \reals^n,$
  we denote $x^\alpha = x_1^{\alpha_1} \dots x_n^{\alpha_n}.$
  
  \section{Wavelet characterisation for the \texorpdfstring{$\BMO$}{BMO}-Sobolev spaces}
  \label{sec:WaveletCharacterisations}
  We begin by indicating how Theorem~\ref{thm:JbmoWavelet} reduces to known results in the literature.
  As mentioned in the Introduction, this characterisation in terms of smooth (not compactly supported) wavelets is a particular case
  of a result of Frazier and Jawerth (see \cite{ref:FrazierJawerthDiscreteTransform}).
  For a general characterisation, one can adapt the techniques of Lemarié and Meyer for the space $\BMO,$
  together with arguments from standard Sobolev spaces theory,
  in order to get the characterisation for the $\bmo$-Sobolev spaces
  (see \cite{ref:LemarieMeyer} and \cite{ref:MeyerWaveletsOperators}).
  Nontheless, for a more self-contained proof, we refer the reader to \cite{ref:TriebelFunctionSpaces-IV}.
  Observe that, due to the John-Niremberg inequality, our definition~\eqref{eq:bmoNorm}
  of the non-homogeneous $\bmo$ norm is equivalent to that appearing in \cite[p.~3]{ref:TriebelFunctionSpaces-IV}.
  Moreover, note as well the coincidence $\Jbmo{s} = \triebel{s}{2}{\infty}$ between the $\bmo$-Sobolev
  and the non-homogeneous Triebel-Lizorkin scales (see Proposition~1.3 and Theorem~1.22 in
  \cite[pp.~4,16]{ref:TriebelFunctionSpaces-IV}).
  Then, the statement of Theorem \ref{thm:JbmoWavelet} is given by Proposition~1.11 (see also Corollary~1.21)
  in \cite[pp.~9,16]{ref:TriebelFunctionSpaces-IV}.
  
  Next, we proof Theorem~\ref{thm:DistanceJbmoWavelets} using the characterisation
  in Theorem~\ref{thm:JbmoWavelet}.
  Recall that, for $s < r,$ the series $\sum_{Q\in\dyadic_0} d_Q \varphi_Q$ belongs to $\Jbmo{s},$
  where $r$ is the regularity of the chosen wavelet basis and the sequence $\{d_Q\}_{Q\in\dyadic_0}$
  is uniformly bounded.
  Hence, we only need to consider the coefficients corresponding to the wavelets $\psi_\omega$ for $\omega\in\mathcal{Q}.$
  Given $0 < s \leq 1$ and an integer $n \geq 1,$ consider the space of sequences
  $a = \{a_\omega\}_{\omega \in \mathcal{Q}}$ with norm
  \begin{equation}
    \label{eq:CoefficientSpaceNorm}
    \norm{a}{s} = \sup_{\omega \in \mathcal{Q}} 2^{|\omega|(n/2+s)} |a_\omega|.
  \end{equation}
  Observe that this is the same as the norm defined in formula~\eqref{eq:ZygmundWaveletCondition}
  of Theorem~\ref{thm:LipWavelet} when $d_Q(f) = 0$ for all $Q \in \dyadic_0.$
  Thus, this sequence space is the space of wavelet coefficients $\{c_\omega\}$ (corresponding to wavelets $\psi_\omega$)
  of functions in $\lip{s}(\reals^n).$
  Moreover, by Theorem~\ref{thm:LipWavelet} we have that $\norm{\{a_\omega\}}{s} \simeq \norm{f}{\lip{s}}$
  where
  \begin{equation*}
    f = \sum_{\omega\in\mathcal{Q}} a_\omega \psi_\omega.
  \end{equation*}
  In the rest of this section, we use $\norm{\cdot}{s}$ both to denote the norm on sequences defined by~\eqref{eq:CoefficientSpaceNorm}
  and the norm equivalent to $\norm{f}{\lip{s}}$ for functions $f \in \lip{s}(\reals^n)$ such that $d_Q(f) = 0$
  for all $Q \in \dyadic_0.$
  \begin{proof}[Proof of Theorem~\ref{thm:DistanceJbmoWavelets}]
    Denote by $\varepsilon_0$ the infimum in~\eqref{eq:DistanceJbmoWavelets}, which we assume to be positive,
    and assume that $\varepsilon < \varepsilon_0.$
    Consider a function $g \in \Jbmo{s}$ and assume that $\norm{f-g}{s} \leq \varepsilon$
    (so that $\norm{f-g}{\lip{s}} \lesssim \varepsilon$).
    Note that we may also assume that $d_Q(g) = d_Q(f)$ for all $Q \in \dyadic_0.$
    Pick $\varepsilon' \in (\varepsilon, \varepsilon_0)$ and note that, for $\omega\in\mathcal{Q},$
    whenever $|c_\omega(f)| > \varepsilon' 2^{-|\omega|(n/2+s)},$ we have that
    $|c_\omega(g)| > \delta 2^{-|\omega|(n/2+s)},$ where $\delta = \varepsilon' - \varepsilon > 0.$
    Thus, for any cube $Q \in \dyadic,$ we have that
    \begin{align*}
      \frac{1}{|Q|}  \sum_{\omega\in\mathcal{Q}(Q)} 4^{s|\omega|}|c_\omega(g)|^2 &\gtrsim
      \frac{\delta^2}{|Q|} \sum_{\substack{P \in W(s,f,\varepsilon')\\ P \subseteq Q}}
      2^{-n\tau(P)} = \frac{\delta^2}{|Q|} \sum_{\substack{P \in W(s,f,\varepsilon')\\ P \subseteq Q}} |P|\\
      &\simeq \frac{\delta^2}{|Q|} \int_Q \int_0^{l(Q)} \chi_{T(s,f,\varepsilon')}(x,y)\, \frac{dy\, dx}{y}.
    \end{align*}
    But the supremum, with $Q$ ranging over all dyadic cubes, of the latter quantity is not finite
    since $\varepsilon' < \varepsilon_0.$
    By Theorem~\ref{thm:JbmoWavelet}, this contradicts that $g \in \Jbmo{s}$ and, thus,
    $\dist_s(f,\Jbmo{s}) \gtrsim \varepsilon_0.$
    
    If $\varepsilon > \varepsilon_0,$ we construct a function $g \in \Jbmo{s}$ such that
    $\dist_s (f,g) \lesssim \varepsilon.$
    Given the wavelet coefficients $\{d_Q(f)\}$ and $\{c_\omega(f)\}$ of $f,$
    we set $d_Q(g) = d_Q(f)$ for every $Q \in \dyadic_0.$
    Clearly, since $f \in \lip{s}$ we have that
    \begin{equation}
        \label{eq:ScalingCoefficientsBound}
        \sup_{Q \in \dyadic_0} |d_Q(g)| < \infty
    \end{equation}
    by Theorem~\ref{thm:LipWavelet}.
    Next, take $c_\omega(g) = c_\omega(f)$ whenever $\omega = (l,P)$ with $P \in W(s,f,\varepsilon),$
    and $c_\omega(g) = 0$ otherwise.
    By construction, we have $\norm{f-g}{s} \leq \varepsilon.$
    Thus, by Theorem~\ref{thm:LipWavelet} we find that $\dist_s (f,g) \lesssim \varepsilon.$
    Furthermore, we have that
    \begin{align*}
      \frac{1}{|Q|}  \sum_{\omega\in\mathcal{Q}(Q)} 4^{|\omega|s}|c_\omega(g)|^2 &\lesssim
      \frac{\norm{c_\omega(f)}{s}^2}{|Q|} \sum_{\substack{P \in W(s,f,\varepsilon)\\ P \subseteq Q}} |P|\\
      &\simeq \frac{\norm{c_\omega(f)}{s}^2}{|Q|} \int_Q \int_0^{l(Q)} \chi_{T(s,f,\varepsilon)}(x,y)\, \frac{dy\, dx}{y}.
    \end{align*}
    Since $\varepsilon > \varepsilon_0,$
    the supremum of the latter quantity when $Q$ ranges over all dyadic cubes is finite.
    Hence, this and~\eqref{eq:ScalingCoefficientsBound} imply that $g \in \Jbmo{s}$ by Theorem~\ref{thm:JbmoWavelet}, as we wanted to show.
  \end{proof}
  
  \section{Properties of the sets \texorpdfstring{$S,$}{S,} \texorpdfstring{$D$}{D} and \texorpdfstring{$T$}{T}}
  \label{sec:PropertiesOfTheSets}
  In the present section, we first estimate local continuity properties of
  the various quantities defined in the upper half-space that we are using to
  quantify the Hölder norm.
  Later on, we use these estimates to control the change in the quantity $M(T(s,f,\varepsilon))$
  and its analogies when the corresponding set is hyperbolically enlarged.
  
  For that purpose,
  let us begin by recalling some basic facts concerning the hyperbolic metric in $\halfspace{n}.$
  The element $ds$ of hyperbolic arc length at $(x,y) \in \halfspace{n}$ is defined by
  \begin{equation*}
    ds^2 = \frac{dx^2 + dy^2}{y^2}.
  \end{equation*}
  Geodesics in this metric are circular arcs intersecting orthogonally the hyperplane $\{y = 0\}$ and vertical lines,
  that is straight lines intersecting orthogonally the same hyperplane.
  We denote by $\rho(a,b)$ the hyperbolic distance between $a,b \in \halfspace{n}$ given by this metric,
  that is the hyperbolic arc length of the geodesic segment joining $a$ and $b.$
  Given a set $A \in \halfspace{n}$ and $R > 0,$ we will consider the $R$-dilation of $A$ in the hyperbolic metric
  given by its hyperbolic $R$-neighbourhood and denote it by $\dilation{A}{R}.$ 
  In other words, we take
  \begin{equation*}
    \dilation{A}{R} = \{p \in \halfspace{n}\colon \rho(p,A) < R\}.
  \end{equation*}
  A hyperbolic ball of radius $r > 0$ and centre $z \in \halfspace{n}$ is denoted by $\hyperball{z}{r}.$
  
  We start by studying how $\Delta_2f(x,y)$ varies.
  \begin{lemma}
    \label{lemma:DifferencesComparisonHolder}
    Let $0 < s < 1.$
    Consider a function $f \in \lip{s}(\reals^n).$
    Then
    \begin{equation*}
      |\Delta_2f(x,y) - \Delta_2f(x',y')| \lesssim \norm{f}{\lip{s}} \left(|x-x'|^s + |y-y'|^s\right).
    \end{equation*}
  \end{lemma}
  \begin{proof}
    Consider an arbitrary auxiliary $p \in \reals^n$ such that $|p| = y$ and note that
    \begin{equation*}
      |(f(x+p)-f(x'+p)) - 2(f(x)-f(x')) + (f(x-p)-f(x'-p))| \lesssim \norm{f}{\lip{s}} |x-x'|^s
    \end{equation*}
    because $f \in \lip{s}(\reals^n).$
    Since this holds uniformly for any such $p,$
    we may use the general fact that for a bounded function $H$ on $\reals^{2n}$ one has that
    \begin{equation}
      \label{eq:BoundedFunctionGeneralFact}
      \left|\sup_{|p|=y}|H(x,p)| - \sup_{|p|=y}|H(x',p)|\right| \leq \sup_{|p|=y}|H(x,p) - H(x',p)|
    \end{equation}
    to deduce that
    \begin{equation}
      \label{eq:EqualScaleDifferencesHolder}
      |\Delta_2f(x,y) - \Delta_2f(x',y)| \lesssim \norm{f}{\lip{s}} |x-x'|^s.
    \end{equation}
    On the other hand, if $q = (y'/y)p,$ we also have that
    \begin{equation*}
      |(f(x'+p)-f(x'+q)) + (f(x'-p)-f(x'-q))| \lesssim \norm{f}{\lip{s}} |p-q|^s = \norm{f}{\lip{s}} |y-y'|^s.
    \end{equation*}
    This is true uniformly for any such $p$ and, thus, it holds that
    \begin{equation}
      \label{eq:EqualCentreDifferencesHolder}
      |\Delta_2f(x',y) - \Delta_2f(x',y')| \lesssim \norm{f}{\lip{s}} |y-y'|^s.
    \end{equation}
    As we join~\eqref{eq:EqualScaleDifferencesHolder} and~\eqref{eq:EqualCentreDifferencesHolder}
    the conclusion follows immediately.
  \end{proof}
  
  We deal with the case $s = 1$ separately.
  
  \begin{lemma}
    \label{lemma:DifferencesComparisonZygmund}
    Consider a function $f \in \zyg(\reals^n) = \lip{1}(\reals^n).$
    If $|x-x'| < (y+y')/2,$ then
    \begin{equation*}
      \begin{split}
        |\Delta_2f(x,y) - \Delta_2f(x',y')| \lesssim \norm{f}{\zyg}
        \bigg(&|x-x'| \log\left(e+\frac{y+y'}{|x-x'|}\right)\\
        + &|y-y'| \log\left(e+\frac{y+y'}{|y-y'|}\right)\bigg).
      \end{split}
    \end{equation*}
  \end{lemma}
  \begin{proof}
    Consider smooth symmetric functions $\widehat{\phi_0}$ and $\widehat{\phi_1}$
    with $\supp{\widehat{\phi_0}} \subseteq \{\xi \in \reals^n\colon |\xi| \leq 1\}$
    and $\supp{\widehat{\phi_1}} \subseteq \{\xi \in \reals^n\colon 1/2 \leq |\xi| \leq 2\},$
    and such that if $\widehat{\phi_j}(\xi) = \widehat{\phi_1}(2^{-(j-1)}\xi)$ for $j \in \naturals,$
    then $\sum_{j\in\naturals_0} \widehat{\phi_j}(\xi) = 1$ for every $\xi \in \reals^n.$
    Now, take $\phi_j = \invfourier{\widehat{\phi_j}}$ and $f_j = \phi_j \ast f$ for $j \in \naturals_0,$
    so that the Littlewood-Paley dyadic decomposition of $f$ can be writen as $f = \sum_{j\in\naturals_0} f_j.$
    It is a well known fact (see for instance \cite[p.~253]{ref:SteinHarmonicAnalysis})
    that, for a given $s > 0,$ $f \in \lip{s}$  if and only if $\norm{f_j}{\Lp{\infty}} \leq C 2^{-js}$ for all $j \geq 0,$ and
    \begin{equation*}
      \norm{f}{\lip{s}} \simeq \sup_{j \geq 0} 2^{js} \norm{f_j}{\Lp{\infty}}.
    \end{equation*}
    Moreover, we have that $\norm{\partial^\alpha f_j}{\Lp{\infty}} \leq \norm{f}{\lip{s}} 2^{j|\alpha|} 2^{-js}$
    for any multi-index $\alpha.$
    In particular, this applies to the Zygmund class taking $s = 1.$
    
    Take $p \in \reals^n$ with $|p| = y.$
    For this particular $p$ we have that
    \begin{multline*}
      |f(x+p) - 2f(x) + f(x-p) - f(x'+p) + 2f(x') - f(x'-p)|\\
      \leq \sum_{j\in\naturals_0} |f_j(x+p) - 2f_j(x) + f_j(x-p) - f_j(x'+p) + 2f_j(x') - f_j(x'-p)|.
    \end{multline*}
    We split this sum into those terms for which $2^j < 1/(y+y'),$ those with $1/(y+y') \leq 2^j < 1/|x-x'|$
    and those with $2^j \geq 1/|x-x'|.$
    For the first part, we express
    \begin{equation*}
      f_j(x+p) - 2f_j(x) + f_j(x-p) = \int_{-1}^{1} (1-|u|) \frac{d^2}{du^2} f_j(x+up)\, du.
    \end{equation*}
    Using the bound on the third derivatives of $f_j$ we obtain
    \begin{equation*}
      \left|\frac{d^2}{du^2}f_j(x+up) - \frac{d^2}{du^2}f_j(x'+up)\right| \lesssim \norm{f}{\zyg} |p|^2 2^{2j}|x-x'|.
    \end{equation*}
    Hence
    \begin{multline*}
      |f_j(x+p) - 2f_j(x) + f_j(x-p) - f_j(x'+p) + 2f_j(x') - f_j(x'-p)|\\
      \lesssim \norm{f}{\zyg} 2^{2j} |x-x'| y^2.
    \end{multline*}
    This yields
    \begin{equation*}
      \begin{split}
        \sum_{2^j < 1/(y+y')} |f_j(x+p) - 2f_j(x) + f_j(x-p) &- f_j(x'+p) + 2f_j(x') - f_j(x'-p)|\\
        &\lesssim \norm{f}{\zyg} |x-x'| y^2 \sum_{2^j < 1/(y+y')} 2^{2j}\\
        &\lesssim \norm{f}{\zyg} |x-x'|.
      \end{split}
    \end{equation*}
    When $1/(y+y') \leq 2^j < 1/|x-x'|,$ we use the $j$-independent uniform bound for
    the first derivative of $f_j$ to obtain directly
    \begin{multline*}
      |f_j(x+p) - 2f_j(x) + f_j(x-p) - f_j(x'+p) + 2f_j(x') - f_j(x'-p)|\\
      \lesssim \norm{f}{\zyg} |x-x'|.
    \end{multline*}
    Then we find that
    \begin{equation*}
      \begin{split}
        \sum_{1/(y+y') \leq 2^j < 1/|x-x'|} |f_j(x+p) - 2f_j(x) &+ f_j(x-p)\\
        &- f_j(x'+p) + 2f_j(x') - f_j(x'-p)|\\
        &\lesssim \norm{f}{\zyg} |x-x'| \sum_{1/(y+y') \leq 2^j < 1/|x-x'|} 1\\
        &\lesssim \norm{f}{\zyg} |x-x'| \log\left(e+\frac{y+y'}{|x-x'|}\right).
      \end{split}
    \end{equation*}
    Finally, using that $\norm{f_j}{\Lp{\infty}} \lesssim \norm{f}{\zyg} 2^{-j}$
    we get that the remaining terms are bounded by
    \begin{equation*}
      \begin{split}
        \sum_{2^j \geq 1/|x-x'|} |f_j(x+p) - 2f_j(x) &+ f_j(x-p)\\
        &- f_j(x'+p) + 2f_j(x') - f_j(x'-p)|\\
        &\lesssim \norm{f}{\zyg} \sum_{2^j \geq 1/|x-x'|} 2^{-j}\\
        &\lesssim \norm{f}{\zyg} |x-x'|.
      \end{split}
    \end{equation*}
    Since all the above bounds are uniform on $p$ with $|p| = y,$
    applying again observation~\eqref{eq:BoundedFunctionGeneralFact}, we obtain the estimate
    \begin{equation}
      \label{eq:EqualScaleDifferencesZygmund}
      |\Delta_2f(x,y) - \Delta_2f(x',y)| \lesssim \norm{f}{\zyg} |x-x'|
      \log\left(e+\frac{y+y'}{|x-x'|}\right).
    \end{equation}
    
    Now, let $p$ be as before and consider the case $x = x'$ but $y \neq y'.$
    We take $q = (y'/y)p,$ and note that it is enough to estimate the quantity
    \begin{multline*}
      |f(x'+p) - f(x'+q) + f(x'-p) - f(x'-q)|\\
      \leq \sum_{j\in\naturals_0} |f_j(x'+p) - f_j(x'+q) + f_j(x'-p) - f_j(x'-q)|.
    \end{multline*}
    We split the previous sum into those terms for which
    $2^j < 1/(y+y'),$ those with $1/(y+y') \leq 2^j < 1/|y-y'|$ and those with $2^j \geq 1/|y-y'|,$
    and follow the previous argument with minor changes.
    Towards estimating  the first sum, we observe first the elementary bound
    \begin{equation*}
      \begin{split}
        |g(1,1)-g(1,-1)+g(-1,-1)-g(-1,1)| &= \left|\int_{[-1,1]^2}g_{uv}(u,v)\, du\, dv\right|\\
        &\leq 4 \sup_{(u,v) \in [-1,1]^2}|g_{uv}(u,v)|
      \end{split}
    \end{equation*}
    As we apply this to $g(u,v) \coloneqq f_j\left(x+u\frac{p+q}{2}+v\frac{p-q}{2}\right),$
    together with the known bound $\norm{\partial^\alpha f_j}{\Lp{\infty}} \lesssim \norm{f}{\zyg} 2^j$
    for multi-indices of length $|\alpha| = 2,$ it follows that
    \begin{equation*}
      |f_j(x'+p) - f_j(x'+q) + f_j(x'-p) - f_j(x'-q)| 
      \lesssim \norm{f}{\zyg} 2^j |y-y'| (y+y'),
    \end{equation*}
    and so
    \begin{equation*}
      \sum_{2^j < 1/(y+y')} |f_j(x'+p) - f_j(x'+q) + f_j(x'-p) - f_j(x'-q)|
      \lesssim \norm{f}{\zyg} |y-y'|.
    \end{equation*}
    Then, for those terms with $1/(y+y') \leq 2^j < 1/|y-y'|,$ using the $j$-uniform Lipschitz property of $f_j$
    we may deduce
    \begin{equation*}
      |f_j(x'+p) - f_j(x'+q) + f_j(x'-p) - f_j(x'-q)| \lesssim \norm{f}{\zyg} |y-y'|,
    \end{equation*}
    which yields
    \begin{multline*}
      \sum_{1/(y+y') \leq 2^j < 1/|y-y'|} |f_j(x'+p) - f_j(x'+q) + f_j(x'-p) - f_j(x'-q)|\\
      \lesssim \norm{f}{\zyg} |y-y'| \log\left(e+\frac{y+y'}{|y-y'|}\right).
    \end{multline*}
    Finally, use the size estimate $\norm{f_j}{\Lp{\infty}} \lesssim \norm{f}{\zyg} 2^{-j}$ to get
    \begin{equation*}
      \sum_{2^j \geq 1/|y-y'|} |f_j(x'+p) - f_j(x'+q) + f_j(x'-p) - f_j(x'-q)|
      \lesssim \norm{f}{\zyg} |y-y'|.
    \end{equation*}
    These bounds are uniform for any $p$ and $q$ such that $|p| = y$ and $q = (y'/y)p$
    and, therefore, it is clear that
    \begin{equation}
      \label{eq:EqualCentreDifferencesZygmund}
      |\Delta_2f(x',y) - \Delta_2f(x',y')| \lesssim \norm{f}{\zyg} |y-y'|
      \log\left(e+\frac{y+y'}{|y-y'|}\right).
    \end{equation}
    The statement of the lemma follows from~\eqref{eq:EqualScaleDifferencesZygmund}
    and~\eqref{eq:EqualCentreDifferencesZygmund}.
  \end{proof}
  
  As a second step we study the variation of $y^2 (\partial^2 \poisson{f}/\partial y^2)(x,y).$
  
  \begin{lemma}
    \label{lemma:DerivativesComparison}
    Let $0 < s \leq 1$ and consider a function $f \in \lip{s}(\reals^n).$
    Denote by $u$ the harmonic extension $\poisson{f}$ of $f$ to $\halfspace{n}.$
    Then, for any $(x,y),(x',y') \in \halfspace{n},$ we have that
    \begin{equation*}
      \left|y^{2-s}\frac{\partial^2 u}{\partial y^2}(x,y)
      - {y'}^{2-s}\frac{\partial^2 u}{\partial y^2}(x',y')\right|
      \lesssim \norm{f}{\lip{s}} \rho\left((x,y),(x',y')\right).
    \end{equation*}
  \end{lemma}
  \begin{proof}
    Recall that
    \begin{equation*}
      \left|y^{2-s}\frac{\partial^2 u}{\partial y^2}\right|
      \lesssim \norm{f}{\lip{s}}.
    \end{equation*}
    Moreover, this is equivalent to
    \begin{equation}
      \label{eq:HigherHyperDerivativesBound}
      \left|y^{l-s}\frac{\partial^l u}{\partial y^l}\right|
      \leq C_l \norm{f}{\lip{s}}
    \end{equation}
    for any integer $l > 2$ (see \cite[Chapter~V]{ref:Stein}), where $C_l$ only depends on $l.$
    Define the function
    \begin{equation*}
      g(x,y) = y^{2-s} \frac{\partial^2 u}{\partial y^2}(x,y)
    \end{equation*}
    and let us denote by $\mathrm{D}g$ the gradient of $g.$
    Then, the hyperbolic derivative of $g$ may be estimated in view of~\eqref{eq:HigherHyperDerivativesBound} as
    \begin{align*}
      |y \mathrm{D}g(x,y)| &\leq y \left(\left|\frac{\partial g}{\partial y}(x,y)\right|
                                        +\sum_{k=1}^n \left|\frac{\partial g}{\partial x_k}(x,y)\right|\right)\\
                           &\leq (2-s) y^{2-s} \left|\frac{\partial^2 u}{\partial y^2}(x,y)\right|
                                 + y^{3-s} \left|\frac{\partial^3 u}{\partial y^3}(x,y)\right|\\
                                 &\quad + y^{3-s} \sum_{k=1}^n \left|\frac{\partial^3 u}{\partial y^2 \partial x_k}(x,y)\right|\\
                           &\lesssim \norm{f}{\lip{s}}.
    \end{align*}
    Hence, $g$ is locally Lipschitz with respect to the hyperbolic metric,
    which also implies the global Lipschitz property as the hyperbolic metric is geodesic.
    This implies the claim.
  \end{proof}
  
  Recall from~\eqref{eq:CarlesonSetMeasure} that for a given measurable set $A \subseteq \halfspace{n}$ we defined the quantity
  \begin{equation}
    \label{eq:DefMQuantity}
    M(A) = \sup_{Q\in\dyadic} \frac{1}{|Q|} \int_Q\int_0^{l(Q)} \chi_{A}(x,y)\, \frac{dy\, dx}{y}.
  \end{equation}
  
  \begin{lemma}
    \label{lemma:Dilations}
    Assume that $A\subset\reals^{n+1}_+$ has the following measure density property:
    there are $\delta,\delta'\in (0,1/10)$ so that the hyperbolic $\delta$-neighbourhood
    of any point $z\in A$ satisfies
    \begin{equation*}
      |\hyperball{z}{\delta} \cap A|\geq \delta'|\hyperball{z}{\delta}|.
    \end{equation*}
    Then, if $M(A)<\infty$ we also have $M(\dilation{A}{R}) < \infty$  for any $R > 0.$
  \end{lemma}
  \begin{proof} 
    For any cube $Q \subset \reals^n$ we denote
    $\widetilde Q \coloneqq Q \times (0,l(Q)) \subset \halfspace{n},$
    and by $Q' = 3Q$ the cube that is concentric to $Q$ and that has $l(Q') = 3 l(Q).$
    We also denote by $\mu$ the Borel measure on $\halfspace{n}$
    with density $d\mu = \chi_{\{0 < y \leq 1\}} y^{-1}\, dx\, dy.$
    Observe that, with this notation at hand, the integral appearing in~\eqref{eq:DefMQuantity} for a given cube $Q \in \dyadic$
    is the same as $\mu(\widetilde{Q} \cap A).$
    Thus, the assumption that $M(A)$ is finite is now equivalent to
    \begin{equation}
      \label{eq:MeasureCondition}
      \mu(\widetilde{Q} \cap A) \leq C|Q|
    \end{equation}
    for all $Q \in \dyadic.$
    Condition~\eqref{eq:MeasureCondition} is stated only for dyadic cubes with $l(Q) \leq 1,$
    but it immediately extends to all cubes of arbitrary size with another constant $C.$
    
    Let us then fix a dyadic cube $Q \subset \reals^n$ with side-length at most $1.$
    Choose (via Zorn Lemma or by an elementary argument) a maximal subset $\mathcal{P}$ of $\widetilde{Q} \cap A$
    such that $d_\rho(z,z') \geq 2\delta$ for any distinct $z,z' \in \mathcal{P}.$
    Then the (open) balls $\hyperball{z}{\delta},$ $z \in \mathcal{P},$ are disjoint and,
    since $\delta < 1/10,$ we have directly by construction and the measure density condition of $A$ that
    \begin{equation*}
      \mu(\hyperball{z}{\delta}) \lesssim (1/\delta') \mu(\hyperball{z}{\delta} \cap A) \quad \textrm{for all } z \in \mathcal{P}.
    \end{equation*}
    Here we used that, for $\delta < 1/10,$ the $\mu$-measure of the hyperbolic ball $\hyperball{z}{\delta}$ is comparable,
    up to a multiplicative constant, to its Lebesgue measure in the upper half-space.
    Clearly $\bigcup_{z \in \mathcal{P}} \hyperball{z}{\delta} \subset \widetilde{Q'}$, whence we obtain
    \begin{equation}
      \label{eq:HyperbolicBallsMuBound}
      \begin{split}
        \sum_{z \in \mathcal{P}} \mu(\hyperball{z}{\delta})
        &\lesssim (1/\delta') \sum_{z \in \mathcal{P}} \mu(\hyperball{z}{\delta} \cap A)\\
        &\leq (1/\delta') \mu(A \cap \widetilde{Q'}) \leq (1/\delta') C |Q'| \lesssim C |Q|.
      \end{split}
    \end{equation}
    For any $R > 0,$ we have that $\dilation{\hyperball{z}{\delta}}{R} = \hyperball{z}{\delta + R}.$
    Thus, since the $\mu$-measure of a hyperbolic ball only depends on its radius,
    we have $\mu(\dilation{\hyperball{z}{\delta}}{R}) \leq c(R) \mu(\hyperball{z}{\delta}),$
    where $c(R) < \infty$ is independent of $z.$
    Since $\mathcal{P}$ is $2\delta$-dense in $\widetilde{Q} \cap A$ we finally infer that
    \begin{align*}
      \mu(\dilation{A}{R} \cap \widetilde{Q}) &\leq \sum_{z \in \mathcal{P}} \mu(\dilation{\hyperball{z}{\delta}}{3\delta+R}) \leq c(R+3\delta) \sum_{z \in \mathcal{P}} \mu(\hyperball{z}{\delta})\\
               &\lesssim c(R+3\delta) C |Q|,
    \end{align*}
    where we applied~\eqref{eq:HyperbolicBallsMuBound} at the last step.
    This shows that $M(\dilation{A}{R}) < \infty.$
  \end{proof}
  
  For simplicity, let us now denote $T(\varepsilon) = T(s,f,\varepsilon),$
  $S(\varepsilon) = S(s,f,\varepsilon)$ and $D(\varepsilon) = D(s,f,\varepsilon).$
  Next, we use Lemma~\ref{lemma:Dilations} to prove useful relations between
  $T(\varepsilon),$ $S(\varepsilon)$ and $D(\varepsilon)$ and their hyperbolic $R$-neighbourhoods.

  \begin{lemma}
    \label{lemma:BadWaveletCubesDilations}
    Let $0 < s \leq 1$ and consider a function $f \in \lip{s}(\reals^n).$
    Denote by $\varepsilon_0 = \inf \{\varepsilon > 0\colon M(T(\varepsilon)) < \infty\}.$
    If $\varepsilon > \varepsilon_0,$ then $M(\dilation{T(\varepsilon)}{R}) < \infty$ for any $R > 0.$
  \end{lemma}
  \begin{proof}
    One simply notes that an arbitrary union of upper halves of Carleson cubes
    obviously satisfies the condition of Lemma~\ref{lemma:Dilations}.
  \end{proof}
  
  \begin{lemma}
    \label{lemma:BadDifferencesDilation}
    Let $0 < s \leq 1$ and consider a function $f \in \lip{s}(\reals^n).$
    Denote by $\varepsilon_0 = \inf \{\varepsilon > 0\colon M(S(\varepsilon)) < \infty\}.$
    If $\varepsilon > \varepsilon_0,$ then $M(\dilation{S(\varepsilon)}{R}) < \infty$ for any $R > 0.$
  \end{lemma}
  \begin{proof}
    Observe that, for $\varepsilon_0 < \varepsilon' < \varepsilon,$ the set $S(\varepsilon)$
    is contained in $S(\varepsilon').$
    Moreover, according to Lemmas~\ref{lemma:DifferencesComparisonHolder} and~\ref{lemma:DifferencesComparisonZygmund}
    (depending on if $s = 1$ or not),
    there exists $\eta > 0$ such that if $|x-x'|/y < \eta$ and $1-\eta < y/y' < 1+\eta,$
    for any $(x,y) \in S(\varepsilon)$ we have that $(x',y') \in S(\varepsilon').$
    That is to say that $S(\varepsilon')$ contains a hyperbolic $\delta$-neighbourhood, for some $\delta > 0,$ of $S(\varepsilon)$
    and since $M(S(\varepsilon')) < \infty$ it is also true that $M(\dilation{S(\varepsilon)}{\delta}) < \infty.$
    By definition, $\dilation{S(\varepsilon)}{\delta}$ may be written as a union of hyperbolic balls of radius $\delta,$
    and hence it clearly satisfies the condition of Lemma~\ref{lemma:Dilations}.
    Especially, $M(\dilation{\dilation{S(\varepsilon)}{\delta}}{R}) < \infty$ for all $R > 0,$ which clearly implies the claim.
  \end{proof}
  
  \begin{lemma}
    \label{lemma:BadDerivativesDilation}
    Let $0 < s \leq 1$ and consider a function $f \in \lip{s}(\reals^n).$
    Denote by $\varepsilon_0 = \inf \{\varepsilon > 0\colon M(D(\varepsilon)) < \infty\}.$
    If $\varepsilon > \varepsilon_0,$ then $M(\dilation{D(\varepsilon)}{R}) < \infty$ for any $R > 0.$
  \end{lemma}
  \begin{proof}
   The proof is exactly the same as in the previous lemma, one just applies instead Lemma~\ref{lemma:DerivativesComparison}.
  \end{proof}
  
  \section{Equivalence of characterisations}
  \label{sec:Equivalences}
  The aim of this section is to prove Theorems~\ref{thm:DistanceJbmoDifferences}
  and~\ref{thm:DistanceJbmoDerivatives}.
  For this purpose we first show some geometric relations between the sets $T(\varepsilon),$
  $S(\varepsilon)$ and $D(\varepsilon).$
  Recall that for a set $A \subset \halfspace{n},$ the set $\dilation{A}{R}$ denotes the hyperbolic $R$-neighbourhood of $A.$
  
  \begin{lemma}
    \label{lemma:DifferenceBoundsWavelets}
    Assume that the regularity of the wavelet basis used to define the set $T(\varepsilon)$
    is $r > n+3.$
    Let $0 < s \leq 1$ and consider a function $f \in \lip{s}(\reals^n).$
    There exists an absolute constant $c > 0$ such that,
    for any $\varepsilon > 0,$ there is $R = R(f,\varepsilon) > 0$ for which
    $T(\varepsilon) \subseteq \dilation{S(c\varepsilon)}{R}.$
  \end{lemma}
  \begin{proof}
    We first note that it is enough to verify for any dyadic cube $Q \subset \reals^n$ with $l(Q) \leq 1$ that,
    if
    \begin{equation}
      \label{eq:DifferenceHypothesis}
      y^{-s} \Delta_2 f(x,y) \leq \varepsilon \quad \textrm{for all } (x,y)\in \dilation{T(Q)}{R},
    \end{equation}
    then the wavelet coefficients $\{c_\omega(f)\}$ of $f$ corresponding to the cube $Q$ will satisfy
    \begin{equation*}
      \sup_{l} |c_{(l,Q)}| \lesssim \varepsilon l(Q)^{n/2+s}, \quad l = 1,\ldots,2^n-1.
    \end{equation*}
    
    First of all, we note that without loss of generality we can assume that $Q = Q_0 = [0,1]^n.$
    Namely, the general result can be reduced to this case by a translation and a rescaling,
    since both the second differences and the wavelet coefficients behave well with respect to
    these operations.
    Recall also that our wavelets $\{\psi_{(l,Q_0)}\}$ have regularity $r > s$ and are compactly supported,
    say $\supp{\psi_{(l,Q_0)}} \subset B(0,R_0)$ for all $l \in \{1,\ldots,2^n-1\}.$
    Moreover, the wavelet functions satisfy
    \begin{equation*}
      \int_{\reals^n} x^\alpha \psi_{(l,Q_0)}(x)\, dx = 0
    \end{equation*}
    for multi-indices of length $0 \leq |\alpha| \leq r,$ or in other words,
    their Fourier transforms satisfy $\partial^\alpha \widehat{\psi_{(l,Q)}}(0) = 0$ for these multi-indices.
    Since $\psi_{(l,Q_0)}$ has compact support, we have that $\widehat{\psi_{(l,Q_0)}} \in \continuous^\infty(\reals^n).$
    Moreover, for any given multi-index $\beta$ the function $x^\beta \psi_{(l,Q_0)}(x) \in \continuous^r(\reals)$ has compact support
    so that $\xi^{\alpha} \mathrm{D}^\beta \widehat{\psi_{(l,Q_0)}}(\xi)$ is bounded for all $|\alpha| \leq r.$
    This implies that
    \begin{equation}
      \label{eq:WaveletsDecay}
      \left|\mathrm{D}^\beta \widehat{\psi_{(l,Q_0)}}(\xi)\right|
      \leq C_\beta (1+|\xi|)^{-r} \quad \textrm{for all multi-indices } \beta.
    \end{equation}
    
    Consider a non-negative and radially symmetric $\continuous^\infty_0$ function $g$
    supported on the annulus $\{x \in \reals^n\colon 1/2 \leq |x| \leq 2\},$ and with integral $1.$
    Then $\widehat{g}(\xi)$ is real, smooth and radially symmetric.
    Moreover, from the fact that $g$ is non-negative and in $\Lp{1}(\reals^n)$ we deduce that
    $\widehat{g}(\xi) \leq \widehat{g}(0) = 1$ for all $\xi \in \reals^n.$
    Furthermore, we get a quantitative version of this statement:
    there is a positive constant $c > 0$ so that
    \begin{equation}
        \label{eq:AuxiliaryPropertyForG}
        1 - \widehat{g}(\xi) \geq c \min(1,|\xi|^2)\quad \textrm{for }\, \xi \in \reals^n.
    \end{equation}
    Indeed, for $\xi$ at a neighbourhood of the origin we have that
    $1-\widehat{g}(\xi) \simeq |\xi|^2$ because $g$ has zero first moments due to its symmetry
    and $\Delta \widehat{g}(0) = -\int_{\reals^n} |x|^2 g(x)\, dx < 0.$
    On the other hand, since $g$ is smooth, we also have
    $1-\widehat{g}(\xi) \simeq 1$ as $\xi \to \infty$ due to the fast decay of $\widehat{g}.$
    Finally, observe that $\widehat{g}(\xi) = 1$ only if $\xi = 0.$
    This follows from the fact that $g$ is a non-negative radially symmetric function
    and that $e^{i2\pi x\cdot\xi}$ is constant as a function of $x$
    on the whole set $\{g \neq 0\}$ only for $\xi = 0,$
    so that for $\xi \neq 0$ we have that
    \begin{equation*}
      \left|\int g(x) e^{i2\pi x\cdot\xi}\, dx\right| < \int g(x)\, dx = \widehat{g}(0) = 1.
    \end{equation*}
    
    Given one of the wavelets $\psi_{(l,Q_0)},$
    we define the function $h$ via its Fourier transform by setting
    \begin{equation*}
      \widehat{h}(\xi)
      = -\frac{\widehat{\psi_l}(\xi)}{\overline{\widehat{g}(0)}-\overline{\widehat{g}(\xi)}}.
    \end{equation*}
    Note that, because of~\eqref{eq:AuxiliaryPropertyForG} and the fact that
    the derivatives of $\widehat{\psi_{(l,Q_0)}}$ vanish up to order at least $r,$
    if we take $r$ large enough, say $r > n+3,$ then $\widehat{h} \in \continuous^{n+1}(\reals^n).$
    Moreover, all the derivatives of $\widehat{h}$ up to order $n+1$ are integrable because of
    the decay~\eqref{eq:WaveletsDecay} of the derivatives of $\widehat{\psi_{(l,Q_0)}}$
    and the uniform boundedness of each derivative of $\widehat{g}.$
    All this implies that $h$ is continuous and $(1+|x|^2)^{(n+1)/2}h(x)$ is bounded.
    In particular, $h$ itself is bounded and integrable.
    
    We are now able to estimate the wavelet coefficient in terms of the second differences:
    \begin{align*}
      \int_{\reals^n} \overline{\psi_{(l,Q_0)}(x)}f(x)\, dx
      &= (2\pi)^{-n} \int_{\reals^n} \overline{\widehat{\psi_{(l,Q_0)}}(\xi)}\widehat{f}(\xi)\, d\xi\\
      &= \int_{\reals^n} \overline{\widehat{h}(\xi)} (\widehat{g}(\xi)-\widehat{g}(0)) \widehat{f}(\xi)\, d\xi.
    \end{align*}
    Note as well that, because the properties of $g$ and $h,$ it holds that
    \begin{align*}
      \int_{\reals^n}\int_{\reals^n} \overline{h(u)}&g(w)f(u+w)\, du\, dw\\
      &= \int_{\reals^n}\int_{\reals^n} \overline{h(u)}g(w)f(u-w)\, du\, dw\\
      &= \int_{\reals^n} \overline{h(u)} (g \ast f)(u)\, du
      = (2\pi)^{-n} \int_{\reals^n} \overline{\widehat{h}(\xi)} \widehat{g}(\xi)\widehat{f}(\xi)\, d\xi
    \end{align*}
    and
    \begin{equation*}
      \int_{\reals^n}\int_{\reals^n} g(w)\overline{h(u)}f(u)\, du\, dw
      = (2\pi)^{-n} \widehat{g}(0) \int_{\reals^n} \overline{\widehat{h}(\xi)}\widehat{f}(\xi)\, d\xi.
    \end{equation*}
    Therefore, noting also that $g$ is bounded and supported in the annulus
    $\{x \in \reals^n\colon 1/2 \leq |x| \leq 2\},$ we obtain
    \begin{align*}
      \bigg|\int &\overline{\psi_l(x)}f(x)\, dx\bigg|\\
      \leq &\frac{1}{2}
      \int_{\reals^n}\int_{\reals^n} |h(u) g(w)|\, |f(u+w)-2f(u)+f(u-w)|\, du\, dw\\
      \leq &c_n
      \int_{(1/2,2)}\int_{\reals^n} |h(u)| \Delta_2f(u,t)\, du\, dt.
    \end{align*}
    Consider the set $A = \{(u,t)\in\halfspace{n}\colon |u| < R/2, 1/2 < t < 2\} \subset \dilation{T(Q_0)}{R},$
    where we can assume~\eqref{eq:DifferenceHypothesis}.
    Note that
    \begin{equation*}
      \iint_{\halfspace{n} \setminus A} |h(u)|\Delta_2f(u,t)\, du\, dt
      \leq \norm{f}{\lip{s}}2^s \int_{|u| \geq R/2} |h(u)|\, du
      \lesssim \varepsilon
    \end{equation*}
    for $R = R(f,\varepsilon)$ large enough.
    On the other hand, we have by~\eqref{eq:DifferenceHypothesis} that
    \begin{equation*}
      \iint_{A} |h(u)|\Delta_2f(u,t)\, du\, dt
      \leq \varepsilon 2^s \int_{|u| \leq R/2 } |h(u)|\, du
      \lesssim \varepsilon.
    \end{equation*}
    And, thus, we get that $|c_{(l,Q_0)}(f)| \lesssim \varepsilon,$ as we wanted to show.
  \end{proof}
  
  Recall that for a function $f \in \lip{s}$ with $0 < s \leq 1,$ and for any integer $k \geq 2,$
  one has that the single condition
  \begin{equation*}
      y^{2-s} \left|\frac{\partial^2 \poisson{f}}{\partial y^2}(x,y)\right|
      \lesssim \norm{f}{\lip{s}}
    \end{equation*}
    is equivalent to
    \begin{equation}
      \label{eq:HyperbolicDerivativesUniformBound}
      y^{k-s} \left|\partial^\alpha \poisson{f} (x,y)\right|
      \leq C_k \norm{f}{\lip{s}},
    \end{equation}
  for all multi-indices $\alpha$ with $|\alpha| = k$
  and with $C_k$ depending only on $k$ (see for example \cite[pp.~143--145]{ref:Stein}).
  Before we establish the analogue of Lemma~\ref{lemma:DifferenceBoundsWavelets}
  for $S(\varepsilon)$ in terms of $\dilation{D(c\varepsilon)}{R},$
  we need the following auxiliary result.
  For the reader's convenience, we provide complete details in the proof
  (and also for analogous estimates later on)
  although many parts of the arguments are well-known for the specialists.
  
  \begin{lemma}
    \label{lemma:HyperbolicDerivativesLocalisation}
    Let $0 < s \leq 1$ and consider a function $f \in \lip{s}$ and an integer $k \geq 2.$
    There exists $R_0 = R_0(f,k) > 0$ such that if $R > R_0$ and
    \begin{equation}
      \label{eq:HyperbolicDerivativeNeighbourhoodBound}
      {y'}^{2-s} \left|\frac{\partial^2 \poisson{f}}{\partial y^2}(x',y')\right|
      \leq \varepsilon,
      \quad (x',y') \in \hyperball{(x,y)}{R},
    \end{equation}
    then
    \begin{equation}
      \label{eq:HyperbolicDerivativePointwiseBound}
      y^{|\alpha|-s} \left|\partial^\alpha \poisson{f} (x,y)\right|
      \lesssim \varepsilon
    \end{equation}
    for every multi-index $\alpha$ with $|\alpha| = k.$
  \end{lemma}
  \begin{proof}
    The arguments we use here are the same as those used to prove
    Lemmas~4 and~5 of \cite[pp.~143--145]{ref:Stein}.
    We may assume that $R \geq 4.$
    We just consider one particular multi-index,
    i.e. we verify that if~\eqref{eq:HyperbolicDerivativeNeighbourhoodBound} holds for $R$ large enough,
    then
    \begin{equation}
      \label{eq:HyperbolicDerivativeParticularIndex}
      y^{2-s} \left|\frac{\partial^2 \poisson{f}}{\partial y \partial x_1} (x,y)\right|
      \lesssim \varepsilon.
    \end{equation}
    The general result then follows by an extension of the argument in this special case.
    
    Let us denote $u(x,y) = \poisson{f}(x,y).$
    Using that for $y > 0$ the Poisson kernel satisfies $P_y(x) = (P_{y/2} \ast P_{y/2})(x),$
    one can express $u(x,y) = (P_{y/2} \ast u_{y/2})(x),$
    where $u_y(t) = u(t,y).$
    Thus, one gets
    \begin{equation*}
      \frac{\partial^3 u}{\partial^2 y \partial x_1} =
      \frac{\partial P_{y/2}}{\partial x_1} \ast \frac{\partial^2 u}{\partial y^2}\bigg|_{y/2}.
    \end{equation*}
    Next write
    \begin{equation*}
      \begin{split}
        \left|\frac{\partial^3 u}{\partial^2 y \partial x_1} (x',y')\right|
        &\leq \int_{|w| \leq (R/4)y'} \left|\frac{\partial P_{y'/2}}{\partial x_1}(w)\right|\,
          \left|\frac{\partial^2 u}{\partial y^2}(x'-w,y'/2)\right|\, dw \\
        &+ \int_{|w| > (R/4)y'} \left|\frac{\partial P_{y'/2}}{\partial x_1}(w)\right|\,
          \left|\frac{\partial^2 u}{\partial y^2}(x'-w,y'/2)\right|\, dw.
      \end{split}
    \end{equation*}
    Thus, since in the first term above we can assume that $(x'-w',y'/2) \in \hyperball{(x,y)}{R},$ we obtain
    \begin{equation*}
      \int_{|w| \leq \widetilde{R}y'} \left|\frac{\partial P_{y'/2}}{\partial x_1}(w)\right|\,
        \left|\frac{\partial^2 u}{\partial y^2}(x'-w,y'/2)\right|\, dw
      \lesssim \varepsilon {y'}^{s-3}
    \end{equation*}
    for the first term, while for the second one 
    it is easily checked by scaling that $\int_{|x|\geq Ay}|\partial P_y/\partial x_1| \lesssim (Ay)^{-1},$ 
    so that
    \begin{equation*}
      \int_{|w| > \widetilde{R}y'} \left|\frac{\partial P_{y'/2}}{\partial x_1}(w)\right|\,
        \left|\frac{\partial^2 u}{\partial y^2}(x'-w,y'/2)\right|\, dw
      \lesssim \frac{\norm{f}{\lip{s}}}{\widetilde{R}} {y'}^{s-3}.
    \end{equation*}
    Summing up, if $\widetilde{R}$ is large enough, we have that
    \begin{equation}
      \label{eq:HigherHyperbolicDerivativeBound}
      \left|\frac{\partial^3 u}{\partial^2 y \partial x_1} (x',y')\right|
      \lesssim \varepsilon {y'}^{s-3}
    \end{equation}
    for all $(x',y') \in \hyperball{(x,y)}{R'},$ for $R' > 0$ possibly smaller than $R,$ but also arbitrarily large.
    
    Now, taking into account that $f$ is uniformly bounded and that the Poisson kernel satisfies
    $\norm{(\partial^2 P_y/\partial y \partial x_1)}{\Lp{1}} \lesssim y^{-2},$
    we have that
    \begin{equation*}
      \left|\frac{\partial^2 u}{\partial y \partial x_1}(x,y)\right| \lesssim y^{-2} \norm{f}{\Lp{\infty}},
    \end{equation*}
    from which it follows that $|(\partial^2 u/\partial y \partial x_1) (x,y)|$ tends to zero as $y \to \infty.$
    Hence, one can express
    \begin{multline*}
      \left|\frac{\partial^2 u}{\partial y \partial x_1}(x,y)\right| \leq
      \int_y^\infty \left|\frac{\partial^3 u}{\partial y^2 \partial x_1}(x,y')\right| \, dy'\\
      = \int_y^{R'y} \left|\frac{\partial^3 u}{\partial y^2 \partial x_1}(x,y')\right| \, dy'
        + \int_{R'y}^\infty \left|\frac{\partial^3 u}{\partial y^2 \partial x_1}(x,y')\right| \, dy'.
    \end{multline*}
    Using~\eqref{eq:HigherHyperbolicDerivativeBound} on the first term,
    and noting that $\rho((x,ty),(x,y)) \leq R'$ for $t \in (1,R'),$ we get the bound
    \begin{equation*}
      \int_y^{R'y} \left|\frac{\partial^3 u}{\partial y^2 \partial x_1}(x,y')\right| \, dy'
      \lesssim \varepsilon \int_y^{R'y} {y'}^{s-3}\, dy' \leq \varepsilon y^{s-2}.
    \end{equation*}
    For the second term, we get
    \begin{equation*}
      \int_{R'y}^\infty \left|\frac{\partial^3 u}{\partial y^2 \partial x_1}(x,y')\right| \, dy'
      \lesssim \norm{f}{\lip{s}} \int_{R'y}^\infty {y'}^{s-3}\, dy'
      = \frac{\norm{f}{\lip{s}}}{{R'}^{2-s}} y^{s-2}
    \end{equation*}
    using the bound $y^{3-s}|(\partial^3 u/\partial y^2 \partial x_1) (x,y)| \lesssim
    \norm{f}{\lip{s}}.$
    Therefore, adding these two bounds, we get~\eqref{eq:HyperbolicDerivativeParticularIndex}
    by choosing $R'$ (and thus also $R$) large enough, depending on $f$ and $\varepsilon,$ as we wanted to see.
  \end{proof}
  
  \begin{lemma}
    \label{lemma:DerivativeBoundsDifferences}
    Let $0 < s \leq 1$ and consider a function $f \in \lip{s}.$
    There exists an absolute constant $c > 0$ such that,
    for any $\varepsilon > 0,$ there is $R = R(f,\varepsilon) > 0$ for which
    $S(\varepsilon) \subseteq \dilation{D(c\varepsilon)}{R}.$
  \end{lemma}
  \begin{proof}
    We may assume that $R \geq 4.$
    Fix $(x,y) \in \halfspace{n},$ and let us denote $u = \poisson{f}.$
    We need to see that, if
    \begin{equation}
      \label{eq:DerivativeHypothesis}
      {y'}^{2-s} \left|\frac{\partial^2 u}{\partial y^2}(x',y')\right|
      \leq \varepsilon
    \end{equation}
    for every $(x',y') \in \hyperball{(x,y)}{R},$ then
    \begin{equation}
      \label{eq:DifferenceConclusion}
      \frac{\Delta_2f(x,y)}{y^s} \lesssim \varepsilon.
    \end{equation}
    
    Let us fix $p = (p_1,\ldots,p_n)$ with $|p| = 1.$
    If $f$ was twice continuously differentiable, we would write
    \begin{equation}
      \label{eq:SecondDifferenceUsingSecondDerivative}
      \begin{split}
        |f(x+yp) - 2f(x) &+ f(x-yp)|\\
        &= \int_0^y\int_{-h}^h \frac{d^2}{dt^2}f(x+ tp)\, dt\, dh\\
        &= \int_0^y\int_{-h}^h
        \left(\sum_{i,j=1}^n p_i p_j \frac{\partial^2 f}{\partial x_i \partial x_j}\right)\,
        dt\, dh.
      \end{split}
    \end{equation}
    Note that, for $f \in \lip{s}$ we can express
    \begin{equation*}
      f(x) = \int_0^y y'\frac{\partial^2 u}{\partial y^2}(x,y')\, dy'
      - y\frac{\partial u}{\partial y}(x,y) + u(x,y)
    \end{equation*}
    for any $y > 0.$
    Thus, we can express the second difference of $f$ as
    \begin{equation}
      \label{eq:SecondDifferenceFromHarmonicExtension}
      \begin{split}
        |f(x+yp) &- 2f(x) + f(x-yp)|\\
        &\leq \int_0^y y'\left|\frac{\partial^2 u}{\partial y^2}(x+yp,y')
        -2\frac{\partial^2 u}{\partial y^2}(x,y')
        +\frac{\partial^2 u}{\partial y^2}(x-yp,y')\right|\, dy'\\
        &+y\left|\frac{\partial u}{\partial y}(x+yp,y)
        -2\frac{\partial u}{\partial y}(x,y)
        +\frac{\partial u}{\partial y}(x-yp,y)\right|\\
        &+|u(x+yp,y) -2u(x,y) +u(x-yp,y)|.
      \end{split}
    \end{equation}
    We focus first on the integral term in~\eqref{eq:SecondDifferenceFromHarmonicExtension}.
    Because of~\eqref{eq:DerivativeHypothesis}, we can assume that
    \begin{equation*}
      \left|\frac{\partial^2 u}{\partial y^2}(x',y')\right| \leq \varepsilon {y'}^{-2+s}
    \end{equation*}
    for $2y/R < y' < y$ and $|x-x'| < Ry/2.$
    Therefore, we have that
    \begin{multline*}
      \int_{2y/R}^y y'\left|\frac{\partial^2 u}{\partial y^2}(x+yp,y')
        -2\frac{\partial^2 u}{\partial y^2}(x,y')
        +\frac{\partial^2 u}{\partial y^2}(x-yp,y')\right|\, dy'\\
      \lesssim \int_{2y/R}^y \varepsilon {y'}^{-1+s}\, dy'
      \lesssim \varepsilon y^s.
    \end{multline*}
    On the other hand, we have that
    \begin{multline*}
      \int_0^{2y/R} y'\left|\frac{\partial^2 u}{\partial y^2}(x+yp,y')
        -2\frac{\partial^2 u}{\partial y^2}(x,y')
        +\frac{\partial^2 u}{\partial y^2}(x-yp,y')\right|\, dy'\\
      \lesssim \norm{f}{\lip{s}} \int_0^{2y/R} {y'}^{-1+s}\, dy'
      \lesssim \frac{\norm{f}{\lip{s}}}{R^s} y^s,
    \end{multline*}
    which will be bounded by $\varepsilon y^s$ for $R$ large enough,
    depending on $f$ and $\varepsilon.$
    
    In order to bound the second and third terms in~\eqref{eq:SecondDifferenceFromHarmonicExtension},
    we express them using~\eqref{eq:SecondDifferenceUsingSecondDerivative}.
    Let $g(t,y) = (\partial u / \partial y)(t,y)$ and note that,
    by Lemma~\ref{lemma:HyperbolicDerivativesLocalisation}, we have
    \begin{equation*}
      \left|\frac{\partial^2 g}{\partial x_i \partial x_j}(x+ tp,y)\right|
      = \left|\frac{\partial^3 u}{\partial x_i \partial x_j \partial y}(x+ tp,y)\right|
      \lesssim \varepsilon y^{-3+s}
    \end{equation*}
    for $|t| < y$ if~\eqref{eq:DerivativeHypothesis} holds for $R$ large enough (independent of $y$).
    Thus,
    \begin{equation*}
      \begin{split}
        y\bigg|\frac{\partial u}{\partial y}(x+yp,y)
        &-2\frac{\partial u}{\partial y}(x,y)
        +\frac{\partial u}{\partial y}(x-yp,y)\bigg|\\
        &\leq y \int_0^{y}\int_{-h}^h \bigg|\frac{d^2}{dt^2}g(x+ tp,y)\bigg|\, dt\, dh\\
        &\lesssim \int_0^y\int_{-h}^h \varepsilon y^{-2+s}\, dt\, dh
        \lesssim \varepsilon y^s.
      \end{split}
    \end{equation*}
    Similarly, to bound the third term in~\eqref{eq:SecondDifferenceFromHarmonicExtension}, we use that
    \begin{equation*}
      \left|\frac{\partial^2 u}{\partial x_i \partial x_j}(x+ tp,y)\right|
      \lesssim \varepsilon y^{-2+s}
    \end{equation*}
    for $|t| < y$ if~\eqref{eq:DerivativeHypothesis} holds for $R$ large enough (independent of $y$),
    again due to Lemma~\ref{lemma:HyperbolicDerivativesLocalisation}.
    The same reasoning as before yields that
    \begin{equation*}
      |u(x+yp,y) -2u(x,y) +u(x-yp,y)| \lesssim \varepsilon y^s.
    \end{equation*}
    This shows that
    \begin{equation*}
      |f(x+yp) - 2f(x) + f(x-yp)| \lesssim \varepsilon y^s
    \end{equation*}
    and, since this bound is uniform on the choice of $p,$ equation~\eqref{eq:DifferenceConclusion} follows.
  \end{proof}
  
  In the following lemma we place the extra condition~\eqref{eq:VanishingScalingCoefficients} on $f,$
  stating that the wavelet coefficients $\{d_Q(f)\}$ of $f$ corresponding to $\varphi_Q,$ $Q \in \dyadic_0,$ all vanish.
  However, this will be irrelevant for our later applications.
  
  \begin{lemma}
    \label{lemma:WaveletBoundsDerivatives}
    Assume that the regularity of the wavelet basis used to define the set $T(\varepsilon)$
    is $r \geq 2.$
    Let $0 < s \leq 1$ and consider a function $f \in \lip{s}$ such that
    \begin{equation}
      \label{eq:VanishingScalingCoefficients}
      d_Q (f) = 0 \quad \textrm{for all } Q \in \dyadic_0.
    \end{equation}
    Then there exists an absolute constant $c > 0$ such that,
    for any $\varepsilon > 0,$ there is $R = R(f,\varepsilon) > 0$ for which
    $D(\varepsilon) \subseteq \dilation{T(c\varepsilon)}{R}.$
  \end{lemma}
  \begin{proof}
    Fix $(x,y) \in \halfspace{n}$ and
    consider the set $G$ of dyadic cubes of the form $Q = \{x'\in\reals^n\colon 2^jx'-k\in[0,1]^n\}$
    such that $y/R < 2^{-j} < yR$ and $|x-2^{-j}k| \lesssim 2^{-j} R,$
    where $R$ is a positive constant to be determined later.
    By the basic properties of the hyperbolic distance, all top half-cubes $T(Q)$ for $Q \in G$
    are included in a hyperbolic neighbourhood of $(x,y)$ in the upper half-space
    whose radius depends only on $R.$
    It is hence enough to verify that, by an appropriate choice of $R,$ the assumption
    \begin{equation}
      \label{eq:WaveletHypothesis}
      \sup_l |c_{(l,Q)}(f)| \leq \varepsilon 2^{-j(n/2+s)}
    \end{equation}
    for every $Q \in G$ implies that
    \begin{equation}
      \label{eq:DerivativeConclussion}
      y^{2-s} \left|\frac{\partial^2 u}{\partial y^2}(x,y)\right| \lesssim \varepsilon.
    \end{equation}
    
    Recall that, by Theorem~\ref{thm:LipWavelet} and assumption~\eqref{eq:VanishingScalingCoefficients},
    we may write
    \begin{equation*}
      f(x) = \sum_{1 \leq l \leq 2^n-1} \sum_{j\geq 0} \sum_{Q\in\dyadic_j}
      c_{(l,Q)}(f) \psi_{(l,Q)}(x)
    \end{equation*}
    with $|c_{(l,Q)}(f)| \lesssim 2^{-j(n/2+s)} \norm{f}{\lip{s}}$ when $Q \in \dyadic_j.$
    Now, for $j \geq 0,$ let us denote
    \begin{equation*}
      f_j(x) = \sum_{1 \leq l \leq 2^n-1} \sum_{Q\in\dyadic_j} c_{(l,Q)}(f) \psi_{(l,Q)}(x),
    \end{equation*}
    and also consider its harmonic extension $u_j = \poisson{f_j}$ on the upper half-space.
    We estimate first the contribution of $u_j$ to~\eqref{eq:DerivativeConclussion}
    for $j$ such that $2^{-j} > yR.$
    Of course, if $y$ is not small enough, this range is empty and
    the corresponding contribution is automatically $0,$
    and the same remark applies to some other cases considered below.
    Note that, by harmonicity, it is enough to bound $|(\partial^2 u_j/\partial x_i^2) (x,y)|$
    for $1 \leq i \leq n.$
    First observe that
    \begin{equation*}
      \left|\frac{\partial^2 u_j}{\partial x_i^2}(x,y)\right|
      = \left|\left(P_y \ast \frac{\partial^2 f_j}{\partial x_i^2}\right)(x,y)\right|
      \leq \norm{P_y}{\Lp{1}} \norm{\frac{\partial^2 f_j}{\partial x_i^2}}{\Lp{\infty}}
      = \norm{\frac{\partial^2 f_j}{\partial x_i^2}}{\Lp{\infty}}.
    \end{equation*}
    Then, using that for any multi-index $\alpha$ of length $|\alpha| = 2$
    we have $|\partial^\alpha \psi_{(l,Q)}| \lesssim 2^{j(n/2+2)}$ for $Q \in \dyadic_j,$
    the bound on the wavelet coefficients $|c_{(l,Q)}|$
    and the bounded overlap of the wavelet functions (due to their compact support),
    we obtain
    \begin{equation*}
      \left|\frac{\partial^2 u_j}{\partial x_i^2}(x,y)\right|
      \lesssim \norm{f}{\lip{s}} 2^{j(2-s)}.
    \end{equation*}
    Thus, summing over $j,$ for $2^{-j} > yR,$ we get that
    \begin{equation}
      \label{eq:DerivativeBoundLowFreq}
      \sum_{2^{-j} > yR} \left|\frac{\partial^2 u_j}{\partial x_i^2}(x,y)\right|
      \lesssim \norm{f}{\lip{s}} y^{-2+s} R^{-2+s} \lesssim \varepsilon y^{-2+s},
    \end{equation}
    where the last inequality holds for $R = R(f,\varepsilon)$ large enough (independent of $(x,y)$),
    since $s \leq 1.$
    
    Next, we compute the contribution of $u_j$ to~\eqref{eq:DerivativeConclussion} for $j$
    such that $2^{-j} \leq y/R.$
    In this case, we have that
    \begin{equation*}
      \left|\frac{\partial^2 u_j}{\partial y^2}(x,y)\right|
      = \left|\left(\frac{\partial^2 P_y}{\partial y^2} \ast f_j\right)(x,y)\right|
      \leq \norm{\frac{\partial^2 P_y}{\partial y^2}}{\Lp{1}} \norm{f_j}{\Lp{\infty}}.
    \end{equation*}
    One can see by direct computation that
    \begin{equation}
      \label{eq:PoissonSecondDerivative}
      \left|\frac{\partial^2 P_y}{\partial y^2}(x,y)\right|
      \lesssim y^{-n-2} \left(1 + \frac{|x|^2}{y^2}\right)^{-(n+3)/2},
    \end{equation}
    so in particular $\norm{(\partial^2 P_y/\partial y^2)}{\Lp{1}} \lesssim 1/y^2.$
    This last estimate together with $\norm{f_j}{\Lp{\infty}} \lesssim \norm{f}{\lip{s}} 2^{-js},$ which holds again
    because of the bound on the wavelet coefficients and the bounded overlap
    of the wavelets themselves, show that
    \begin{equation*}
      \left|\frac{\partial^2 u_j}{\partial y^2}(x,y)\right|
      \lesssim \norm{f}{\lip{s}} y^{-2} 2^{-js}.
    \end{equation*}
    Summing now over $j,$ for $2^{-j} \leq y/R,$ we get that
    \begin{equation}
      \label{eq:DerivativeBoundHighFreq}
      \sum_{2^{-j} \leq y/R} \left|\frac{\partial^2 u_j}{\partial y^2}(x,y)\right|
      \lesssim \norm{f}{\lip{s}} y^{-2+s} R^{-s} \lesssim \varepsilon y^{-2+s},
    \end{equation}
    where the last inequality holds for $R = R(f,\varepsilon)$ large enough.
    
    For $j$ such that $y/R < 2^{-j} \leq yR,$ we express $f_j = g_j + h_j,$ where
    \begin{equation*}
      g_j(x) = \sum_{1 \leq l \leq 2^n-1} \sum_{Q\in\dyadic_j\cap G} c_{(l,Q)}(f) \psi_{(l,Q)}(x).
    \end{equation*}
    If $y < 2^{-j} \leq yR,$ as we did in the case $yR < 2^{-j},$ we have that
    \begin{multline*}
      \left|\frac{\partial^2 u_j}{\partial x_i^2}(x,y)\right|
      = \left|\left(P_y \ast \frac{\partial^2 f_j}{\partial x_i^2}\right)(x,y)\right|\\
      \leq \left|\left(P_y \ast \frac{\partial^2 g_j}{\partial x_i^2}\right)(x,y)\right|
      + \left|\left(P_y \ast \frac{\partial^2 h_j}{\partial x_i^2}\right)(x,y)\right|.
    \end{multline*}
    Because of~\eqref{eq:WaveletHypothesis}, the first term is bounded by
    $C \varepsilon 2^{j(2-s)}.$
    Observe that function $h_j$ only contains wavelets whose supports lie on the set
    $\{t\in\reals^n\colon |t-x| \gtrsim yR\}.$
    Thus, we can bound the second term by
    \begin{equation*}
      C \norm{\frac{\partial^2 h_j}{\partial x_i^2}}{\Lp{\infty}} \int_{|t| \gtrsim yR} P_y(t)\, dt
      \lesssim \norm{f}{\lip{s}} 2^{j(2-s)} \frac{1}{R}.
    \end{equation*}
    This yields, by harmonicity, that
    \begin{equation}
      \label{eq:DerivativeBoundMedLowFreq}
      \sum_{y < 2^{-j} \leq yR} \left|\frac{\partial^2 u_j}{\partial y^2}(x,y)\right|
      \lesssim \left(\varepsilon + \frac{\norm{f}{\lip{s}}}{R}\right) y^{-2+s}
      \lesssim \varepsilon y^{-2+s},
    \end{equation}
    where the last inequality holds for $R = R(f,\varepsilon)$ large enough.
    Similarly, if $y/R < 2^{-j} \leq y,$ we write
    \begin{multline*}
      \left|\frac{\partial^2 u_j}{\partial y^2}(x,y)\right|
      = \left|\left(\frac{\partial^2 P_y}{\partial y^2} \ast f_j\right)(x,y)\right|\\
      \leq \left|\left(\frac{\partial^2 P_y}{\partial y^2} \ast g_j\right)(x,y)\right|
      + \left|\left(\frac{\partial^2 P_y}{\partial y^2} \ast h_j\right)(x,y)\right|.
    \end{multline*}
    Now, the first term is bounded by $C \varepsilon y^{-2} 2^{-js}$ because of condition~
    \eqref{eq:WaveletHypothesis}.
    Taking into account that the wavelets appearing in $h_j$ are supported on
    $\{t\in\reals^n\colon |t-x| \gtrsim 2^{-j} R\},$ the second term is bounded by
    \begin{equation*}
      C \norm{h_j}{\infty} \int_{|t| \gtrsim 2^{-j} R} \left|\frac{\partial^2 P_y}{\partial y^2}(t)\right|\, dt
      \lesssim \norm{f}{\lip{s}} y 2^{j(3-s)} \frac{1}{R^3},
    \end{equation*}
    where we have used equation~\eqref{eq:PoissonSecondDerivative} to see that
    $\int_{|x|>Ay} |(\partial^2 P_y/\partial y^2)(t)\, dt| \lesssim y^{-2}A^{-3}.$
    It follows that
    \begin{equation}
      \label{eq:DerivativeBoundMedHighFreq}
      \sum_{y/R < 2^{-j} \leq y} \left|\frac{\partial^2 u_j}{\partial y^2}(x,y)\right|
      \lesssim \left(\varepsilon + \frac{\norm{f}{\lip{s}}}{R^s}\right) y^{-2+s}
      \lesssim \varepsilon y^{-2+s},
    \end{equation}
    where the last inequality holds for $R = R(f,\varepsilon)$ large enough
    (and independent of $(x,y)$).
    Since $u = \sum_j u_j,$ estimates~\eqref{eq:DerivativeBoundLowFreq},~
    \eqref{eq:DerivativeBoundHighFreq},~\eqref{eq:DerivativeBoundMedLowFreq}
    and~\eqref{eq:DerivativeBoundMedHighFreq} yield~\eqref{eq:DerivativeConclussion},
    as we wanted to see.
  \end{proof}
  
  \begin{proof}[Proof of Theorems~\ref{thm:DistanceJbmoDifferences} and~\ref{thm:DistanceJbmoDerivatives}]
    Let $0 < s \leq 1$ and fix $f \in \lip{s}.$
    Let us denote by $\tau_0 = \tau_0(f)$ the infimum in~\eqref{eq:DistanceJbmoWavelets},
    by $\sigma_0 = \sigma_0(f)$ the infimum in~\eqref{eq:DistanceJbmoDifferences} and
    by $\delta_0 = \delta_0(f)$ the one in~\eqref{eq:DistanceJbmoDerivatives}.
    Also, let $c$ be the smallest of the constants appearing in Lemmas~\ref{lemma:DifferenceBoundsWavelets},~
    \ref{lemma:DerivativeBoundsDifferences} and~\ref{lemma:WaveletBoundsDerivatives}.
    Note that we can assume that $c \leq 1.$
    We assume as well that the wavelet basis used in Theorem~\ref{thm:DistanceJbmoWavelets}
    has regularity $r > n + 3$
    -- note that this is just for the sake of the proof below,
    in the actual wavelet characterisation in Theorem~\ref{thm:DistanceJbmoWavelets}
    it is enough to assume that $r > s.$
  
    We first show that
    \begin{equation}
      \label{eq:TauLesssimSigma}
      \tau_0 \leq c^{-1} \sigma_0.
    \end{equation}
    To that end, note that by Lemma~\ref{lemma:BadDifferencesDilation} we have that
    $M(\dilation{S(\varepsilon)}{R})   < \infty$ for all $\varepsilon > \sigma_0$ and $R \geq 1.$
    Then, Lemma~\ref{lemma:DifferenceBoundsWavelets} implies that $T(c^{-1}\varepsilon) < \infty,$
    which gives~\eqref{eq:TauLesssimSigma}.
    Exactly in the same way, Lemmas~\ref{lemma:BadDerivativesDilation} and~\ref{lemma:DerivativeBoundsDifferences}
    yield the inequality
    \begin{equation}
      \label{eq:SigmaLesssimDelta}
      \sigma_0 \leq c^{-1} \delta_0.
    \end{equation}
    Finally, in order to treat the remaining inequality, we write
    \begin{equation*}
      f = \sum_{Q \in \dyadic_0} d_Q(f) \varphi_Q + \sum_{\omega \in \mathcal{Q}} c_\omega(f) \psi_\omega
        = g + b
    \end{equation*}
    and denote $u_g$ (resp. $u_b$) the Poisson extension of $g$ (resp. of $b$).
    Because of the (at least) $\continuous^2$-regularity of the wavelets $\varphi_Q$
    and their bounded overlap due to their compact support,
    we deduce that $\norm{\partial^\alpha g}{\Lp{\infty}} \leq C$ for every multi-index with $|\alpha| \leq 2.$
    This implies that
    \begin{equation*}
      y^{2-s} \left|\frac{\partial^2 u_g}{\partial y^2}\right|
      \lesssim \min (y^{2-s},y^{-s}),
    \end{equation*}
    because of the estimate $\norm{\partial^2 P_y/\partial y^2}{\Lp{1}} \lesssim y^{-2}$
    and the fact that $\norm{P_y}{\Lp{1}} = 1.$
    Since the previous bound tends to zero uniformly as $y \to 0^+$ and as $y \to \infty,$
    we deduce immediately from the definition that $M(D(s,g,\varepsilon)) < \infty$ for every $\varepsilon > 0.$
    Now, for any $\varepsilon_1 + \varepsilon_2 \leq \varepsilon$ we have that
    \begin{equation*}
      M(D(s,f,\varepsilon)) \leq M(D(s,g,\varepsilon_1)) + M(D(s,b,\varepsilon_2)),
    \end{equation*}
    so we deduce that $\delta_0(f) \leq \delta_0(b).$
    On the other hand, we have that $\tau_0(f) = \tau_0(b)$ by definition,
    and Lemma~\ref{lemma:WaveletBoundsDerivatives} applies to the function $b,$
    so that with this and Lemma~\ref{lemma:BadWaveletCubesDilations} we deduce as before that
    \begin{equation}
      \label{eq:DeltaLesssimTau}
      \delta_0 \leq \delta_0(b) \leq c^{-1} \tau_0(b) = c^{-1} \tau_0.
    \end{equation}
    The proof of Theorems~\ref{thm:DistanceJbmoDifferences} and~\ref{thm:DistanceJbmoDerivatives}
    now follows immediately from Theorem~\ref{thm:DistanceJbmoWavelets}
    and inequalities~\eqref{eq:TauLesssimSigma},~\eqref{eq:SigmaLesssimDelta} and~\eqref{eq:DeltaLesssimTau}.
  \end{proof}
  
  \printbibliography
  
  \Addresses

\end{document}